\documentclass[11pt, a4paper, twoside]{article}
\usepackage[a4paper]{anysize}
\usepackage[english]{babel}
\usepackage[latin1]{inputenc}

\usepackage{amssymb}
\usepackage{amsmath}
\usepackage{amsfonts, amsthm}
\usepackage{enumerate}
\usepackage{xcolor}
\usepackage{bbm}
\usepackage[square,numbers]{natbib}
\newcommand{\ID}{\mathbb{D}}
\newcommand{\IE}{\mathbb{E}}
\newcommand{\E}{\mathbb{E}}

\newcommand{\IL}{\mathbb{L}}
\newcommand{\IN}{\mathbb{N}}

\newcommand{\IP}{\mathbb{P}}
\newcommand{\IQ}{\mathbb{Q}}

\newcommand{\IR}{\mathbb{R}}
\newcommand{\R}{\mathbb{R}}

\newcommand{\cE}{\mathcal{E}}
\newcommand{\cF}{\mathcal{F}}

\newcommand{\cH}{\mathcal{H}}

\newcommand{\cS}{\mathcal{S}}

\newcommand{\barr}{{\bar{r}}}
\newcommand{\barq}{{\bar{q}}}

\newcommand{\ud}{\mathrm{d}}
\newcommand{\uds}{\mathrm{d}s}
\newcommand{\udt}{\mathrm{d}t}

\newcommand{\udws}{\mathrm{d}W_s}

\newcommand{\1}{\mathbbm{1}}
\newcommand{\I}{\int_{t}^T}
\newcommand{\ti}{{t_{i}}}
\newcommand{\tip}{{t_{i+1}}}

\newcommand{\bit}{\begin{itemize}}
\newcommand{\eit}{\end{itemize}}

\theoremstyle{plain}
\newtheorem{theo}{Theorem}[section]
\newtheorem{lemma}[theo]{Lemma}

\newtheorem{coro}[theo]{Corollary}

\theoremstyle{definition}

\newtheorem{remark}[theo]{Remark}

\title{Path regularity and explicit convergence rate for BSDE with truncated quadratic growth
\footnote{Supported at different times by the DFG research center
MATHEON at Berlin, the European research network AMaMeF and by the
``Programa Operacional Ci\^encia e Inova{\c c}{\~a}o 2010 (POCI
2010)'' of the Portuguese Ministry of Science, Technology and Higher
Education, with support from the European Social Fund of the
European Union (EU).}}
\author{\normalsize Peter Imkeller \\[8pt]
        \small  Institut f\"ur Mathematik  \\
        \small  Humboldt-Universit\"at zu Berlin \\
        \small  Unter den Linden 6\\
        \small  10099 Berlin \\
        \small  imkeller@math.hu-berlin.de
        \and
                \normalsize Gon\c calo Dos Reis \\[8pt]
        \small  CMAP \\
        \small  \'Ecole Polytechnique \\
        \small  Route de Saclay \\
        \small  91128 Palaiseau Cedex \\
        \small  dosreis@cmap.polytechnique.fr
\vspace*{0.8cm}}
\date{ \today }

\begin{document}

\selectlanguage{english}

\maketitle

\begin{abstract}
We consider backward stochastic differential equations with drivers of quadratic growth (qgBSDE).
We prove several statements concerning path regularity and stochastic smoothness of the solution processes of the qgBSDE,
in particular we prove an extension of Zhang's path regularity theorem to the quadratic growth setting. We give explicit
convergence rates for the difference between the solution of a qgBSDE and its truncation, filling an important gap in numerics for qgBSDE.
We give an alternative proof of second order Malliavin differentiability for BSDE with drivers that are Lipschitz continuous (and differentiable),
and then derive an analogous result for qgBSDE.
\end{abstract}


%60H07 Stochastic calculus of variations and the Malliavin calculus
%60H10 Stochastic ordinary differential equations
%65C30 Stochastic differential and integral equations


\section{Introduction}

Backward Stochastic differential equations (BSDE) have been
receiving much attention in the last 15 years, due to their central
significance in optimization problems for instance in stochastic
finance, and more generally in stochastic control theory. A
particularly important class, BSDE with drivers of quadratic growth (qgBSDE) introduced in \cite{00Kob},
for example arise in the context of utility optimization problems
with exponential utility functions, or alternatively in questions
related to risk minimization for the entropic risk measure. BSDE
provide the genuinely stochastic approach of control problems which
find their analytical expression in the Hamilton-Jacobi-Bellman
formalism. BSDE with drivers of this type keep being a source of
intensive research.

As for Monte-Carlo methods to simulate random processes, numerical
schemes for BSDE provide a robust method for simulating and
approximating solutions of control problems. Much has been done in
recent years to create schemes for BSDE with Lipschitz continuous
drivers (see \cite{04bouchardtouzi}, \cite{GLW05} or \cite{phd-elie} and
references therein). So far BSDE with drivers of quadratic growth
resisted attempts to allow such schemes, which was the main
motivation for this paper.

If the driver is Lipschitz continuous, following
\cite{04bouchardtouzi}, the strategy to prove convergence of a
numerical discretization combines two ingredients: regularity of the
trajectories of the control process, and a convenient a priori
estimate for the solution. The regularity result we refer to can be found in \cite{phd-zhang} or \cite{Zhang2004}. It allows to establish the convergence order for the approximation of the control process.

Our approach for the case of drivers with quadratic growth consists in adding Zhang's path regularity result to the toolbox of qgBSDEs and, independently of the extension, to answer the question of explicit convergence rates for the truncation procedure in the setting of qgBSDEs.

In a first step, we extend the path regularity result for the
control process to the setting of qgBSDE. The methods we apply to
achieve this goal rely crucially on the power of the stochastic
calculus of variations. If $(Y,Z)$ is the solution pair of a BSDE,
it is well known that the trace of the first Malliavin derivative
allows a description of $Z$ by the formula $D_t\,Y_t=Z_t$, which in
turn allows estimates of $Z$ in the $\sup$ norm, provided an extra
continuity result is established. To describe path regularity of $Z$
efficiently, one also needs estimates of the Malliavin derivative of
$Z$ in the $\sup$ norm, whence second order Malliavin derivatives of
$Y$ are needed and add to the complexity of the treatment. We are able to derive the path regularity result without assuming hypothesis that imply boundedness of the $Z$ process.

In the second step of our approach, we truncate the quadratic growth
part of the driver to fall back into the setting of Lipschitz
continuous drivers. We are able to explicitly capture the
convergence rate for the solutions of the truncated BSDE as a
function of the truncation height. Combining the error estimate for
the truncation with the ones for the discretization in any existent
numerical scheme for BSDE with Lipschitz continuous drivers, we find
a numerical approximation for quadratic growth BSDE. This result does not depend on Zhang's path regularity result but depend partially on the results that lead to it.

This result is new to the best of our knowledge. The truncation procedure, however, does not look like the most efficient solution one hopes for.
The main drawback of the approach resides in the running times of the numerical algorithm. Roughly, if $K$ is the truncation dependent Lipschitz constant,
the time step $h$ of the partition for the usual numerical discretization has to satisfy $e^{K}h<1$ modulo some multiplicative constant which results from
the use of Gronwall's inequality. So if the truncation height increases, $h$ will have to become small very quickly, which computationally is a rather inconvenient fact.
At this stage we have to leave the question open if a method exists with a convergence rate that depends on the Lipschitz constant only in a polynomial fashion instead of an exponential one.
Of course it is conceivable that such a method is based on a discretization of the underlying qgBSDE without the intermediate step of truncating the driver.
However, we wish to point out that such a procedure has its difficulties. From our experience, the discretization may be well defined and studied
as the partition's mesh size tends to zero. But to show convergence to the original solution and to provide a convergence
rate appear as very difficult problems that to date remain unsolved.

The paper is organized as follows. In the introductory Section 2 we recall some of the well known results concerning SDE and BSDE. In section 3 we establish some estimates concerning a special class of BSDE, and in Section 4 we establish the second order Malliavin differentiability of solutions of Lipschitz BSDE and qgBSDE. These results are used in Section 5 to state and prove several regularity results for the trajectories of the solution processes. In Section 6
we discuss convergence rates of solutions of truncated BSDE to those
related to BSDE with drivers of quadratic growth.

\section{Preliminaries}

\subsection{Spaces and Notation}

Throughout fix $T>0$. We work on a canonical Wiener space $(\Omega, \cF,  \IP)$ carrying a
$d$-dimensional Wiener process $W = (W^1,\cdots, W^d)$ restricted to
the time interval $[0,T]$, and we denote by
$\cF=(\cF_t)_{t\in[0,T]}$ its natural filtration enlarged in the
usual way by the $\IP$-zero sets. We shall need the following
operators, and auxiliary spaces of functions and stochastic
processes: let $p\geq 2, m, n, d\in \IN$, $\IQ$ a probability
measure on $(\Omega, \cF)$. We use the symbol $\E^\IQ$ for the
expectation with respect to $\IQ$, and omit the superscript for the
canonical measure $\IP$. For vectors $x = (x^1,\cdots, x^m)$ in
Euclidean space $\R^m$ we write $|x| = (\sum_{i=1}^m
(x^i)^2)^{\frac{1}{2}}$. By $\1_A$ we denote the indicator function of a set $A$. We denote further
\begin{itemize}
\item $C^k_b(\R^m)$ the set of $k$-times differentiable real valued maps defined on $\R^m$ with bounded partial derivatives up to order $k$, and $C^\infty_b(\R^m)=\cap_{k\geq 1} C_b^k(\R^m)$; We omit the subscript $b$ to denote the same set but without the boundedness assumptions.

\item ${B}_n^{m\times d}$ the set of all functions $h:[0,T]\times \IR^n\to\IR^{m\times d}$
for which there is a constant $C$ such that for all $t\in[0,T]$ we have $|h(t,x)|\leq C(1+|x|)$ and $x\mapsto h(t,x)$ is differentiable with bounded Lipschitz derivative;

\item $L^p(\R^m; \IQ)$ the space of $\cF_T$-measurable random
variables $X:\Omega\mapsto\R^m$, normed by $\lVert X\lVert_{L^p}=\E^\IQ[ \, |X|^p]^{\frac{1}{p}}$; $L^\infty$ the space of bounded random variables;

\item  $\cS^p(\R^m)$ the space of all measurable processes $(Y_t)_{t\in[0,T]}$ with values in
$\R^m$ normed by $\| Y \|_{\cS^p} = \E[\left( \sup_{t \in [0,T]}
|Y_t| \right)^{p}]^{\frac{1}{p}}$; $\cS^\infty(\R^m)$ the space of bounded measurable processes;

\item $\cH^p(\R^m, \IQ)$ the space of all progressively measurable processes $(Z_t)_{t\in[0,T]}$ with values in $\R^m$ normed by $\|Z\|_{\cH^p} = \E^\IQ[\left( \int_0^T |Z_s|^2 \ud s \right)^{p/2} ]^{\frac{1}{p}};$

\item $BMO(\IQ)$ or $BMO_2(\IQ)$ the space of square integrable martingales $\Phi$ with $\Phi_0=0$ and satisfying
\[\lVert \Phi \lVert_{BMO(\IQ)}^2= \sup_{\tau}\Big\|\, \E^\IQ\big[ \langle \Phi\rangle_T - \langle \Phi \rangle_\tau| \cF_\tau \big]\Big\|_{\infty}< \infty,\] where the supremum is taken over all stopping times $\tau\in[0,T]$.

\item $\ID^{k,p}(\IR^d)$ and $\IL_{k,d}(\IR^d)$ are the spaces of Malliavin differentiable random variables and processes, see subsection \ref{subsec:malldiff}.
\end{itemize}
If there is no ambiguity about the underlying spaces or measures, we
also omit them as arguments in the function spaces defined above.

To denote stochastic integral processes of the Wiener process on
$[0,T]$, according to Paul-Andr\'e Meyer, we write  \[ Z*W =
\int_0^\cdot Z_s\udws,\quad \textrm{ with }Z\in \cH^2.\]

Constants appearing in inequalities of our proofs will for
simplicity be denoted by $C$, although they may change from line to
line.

\subsection{Malliavin Calculus}\label{subsec:malldiff}

We shall use techniques of the stochastic calculus of variations. To this end, we use the following notation.
For more details, we refer the reader to \cite{Nualart}.
Let ${\bf \cS}$ be the space of random variables of the form
\[
\xi = F\Big((\int_0^T h^{1,i}_s \ud W^1_s)_{1\le i\le
n},\cdots,(\int_0^T h^{d,i}_s \ud W^d_s)_{1\le i\le n})\Big),
\] where $F\in C_b^\infty(\R^{n\times d})$, $h^1,\cdots,h^n\in L^2([0,T]; \R^d)$, $n\in\IN.$
To simplify the notation, we assume that all $h^j$ are written as row vectors.
For $\xi\in {\bf \cS}$, we define $D = (D^1,\cdots, D^d):{\bf \cS}\to L^2(\Omega\times[0,T])^d$ by
\[
D^i_\theta \xi = \sum_{j=1}^n \frac{\partial F}{\partial x_{i,j}}
\Big( \int_0^T h^1_t \ud W_t,\ldots,\int_0^T
h^n_t\ud W_t\Big)h^{i,j}_\theta,\quad 0\leq \theta\leq T,\quad 1\le
i\le d,\]
and for $k\in\IN$ and $\theta = (\theta_1,\cdots, \theta_k)\in[0,T]^k$ its $k$-fold iteration by
\[
D_\theta^{(k)} = (D^{i_1}_{\theta_1}\cdots D^{i_k}_{\theta_k})_{1\le i_1,\cdots, i_k\le d\,}.
\]
For $k\in\IN, p\ge 1$ let $\ID^{k,p}$ be the closure of $\cS$ with respect to the norm
\[
\lVert \xi \lVert_{k,p}^p = \Big\{ \|\xi\|^p_{L^p}+ \sum_{i=1}^{k}\|\,| D^{(i)} \xi|\, \|_{(\cH^p)^i}^p \Big\}.
\]
$D^{(k)}$ is a closed linear operator on the space $\mathbb{D}^{k,p}$. Observe
that if $\xi\in \ID^{1,2}$ is $\cF_t$-measurable then $D_\theta \xi=0$ for $\theta
\in (t,T]$. Further denote $\mathbb{D}^{k,\infty}=\cap_{p>1}\mathbb{D}^{k,p}$.

We also need Malliavin's calculus for smooth stochastic processes with values in $\R^m.$ For $k\in\IN, p\ge 1,$ denote by $\mathbb{L}_{k,p}(\R^m)$ the
set of $\R^m$-valued progressively measurable processes
$u = (u^1,\cdots, u^m)$ on $[0,T]\times \Omega$ such that
\begin{itemize}
\item[i)]
For Lebesgue a.a. $t\in[0,T]$, $u(t,\cdot)\in(\mathbb{D}^{k,p})^m$;
\item[ii)] $[0,T]\times \Omega \ni (t,\omega)\mapsto D^{(k)} u(t,\omega)\in (L^2([0,T]^{1+k}))^{d\times n}$ admits a progressively measurable
version;
\item[iii)] $\lVert u \lVert_{k,p}^p= \Big\{\, \|\,|u|\,\|_{\cH^p}^p+\sum_{i=1}^{k} \| \,|D^{(i)} u|\,\|_{(\cH^p)^{1+i}}^p\,
\Big\}<\infty$.
\end{itemize}

For instance, for a process $X\in \mathbb{L}_{2,2}(\IR)$ we have
\begin{align*}
\lVert X \lVert_{1,2}^2&= \E\Big[\int_0^T |X_t|^2\ud t+ \int_0^T\int_0^T|D_\theta X_t|^2\ud \theta\ud t \Big],\\
\lVert X \lVert_{2,2}^2&= \lVert X \lVert_{1,2}^2\,+\,\IE\Big[\int_0^T\int_0^T\int_0^T |D_{\theta_1}D_{\theta_2}X_t |^2\ud \theta_1\ud \theta_2\ud t \Big].
\end{align*}
Note that Jensen's inequality gives for all $p\geq 2$
\[
\IE\Big[\Big( \int_0^T \int_0^T|D_u X_t|^2\ud u\,\ud t\Big)^{\frac{p}{2}} \Big]
\leq
T^{p/2-1}\int_0^T  \| D_u X\|_{\cH^p}^p
\ud u.
\]
This inequality is very useful since the techniques used to deal with
BSDE don't allow a direct estimate of the left hand side, but easily
give access to the right hand side.

Occasionally we shall work with processes taking their values
already in a Hilbert space, for instance if we talk about Malliavin
derivatives. We therefore have to generalize the Sobolev spaces
defined above somewhat. For a Hilbert space $H$ we start with elementary
$H$-valued variables of the form  $\xi = \sum_{i=1}^n \xi_i h_i$,
where for $1\le i \le n$ the variable $\xi_i$ is of the form
discussed above, and $h_i \in H.$ We define similarly $D_\theta^j
\xi = \sum_{i=1}^n D_\theta^j \xi_i h_i, 0\le \theta\le T, 1\le j\le
d,$ and higher derivatives by iteration. For $k\in\IN, p\ge 1$ we
then let $\ID^{k,p}(H)$ be the closure of this set of elementary
processes with respect to the norm
\[
\lVert \xi \lVert_{k,p,H}^p = \Big\{{\|\,|\xi|_H\|}^p_{L^p}+
\sum_{i=1}^{k}\|\,| D^{(i)} \xi|\, \|_{H\otimes(\cH^p)^i}^p\Big\}.
\]
$D^{(k)}$ is a closed linear operator on the space
$\mathbb{D}^{k,p}(H)$. In a similar way we define the spaces
$\mathbb{L}_{k,p}(\R^m\times H)$.

We state an extension of Lemma 1.2.3 from \cite{Nualart}. This extension will play a crucial role in our proof of Malliavin differentiability.
\begin{lemma}\label{lemma.from.nualart.extended}
Let $H$ be a Hilbert space, $(F^n)_{n\geq 1}$ a sequence of random
variables with values in $H$ that converges to an $H$-valued process
$F$ in $L^2(\Omega\times H)$ and such that
\[
\sup_{n\in\mathbb{N}}\ %\E[\,\lVert |D F^n|_H \lVert_{\cH^2(\Omega\times[0,T])}]
\lVert |D F^n| \lVert_{H\otimes \cH^2(\Omega\times[0,T])}
<\infty.
\] Then $F$ belongs to $\mathbb{L}^{1,2}(\IR\times H)$,
and the sequence of derivatives $(D F^n)_{n\in\IN}$ converges to $D
F$ in the weak topology of $H\otimes \cH^2(\Omega\times [0,T])$.
\end{lemma}
\begin{proof}
The proof of this lemma is analogous to the proof of Lemma 1.2.3 of
\cite{Nualart}, being based on the closedness of the Malliavin
derivative operator.
\end{proof}
%\begin{lemma}[Lemma 1.2.3 in \cite{Nualart}]\label{lemma.from.nualart}
%Let $\{F_n,\,n\geq 1\}$ be a sequence of random variables in $\mathbb{D}^{1,2}$ that converges to $F$ in $L^2(\Omega)$ and such that
%\[\sup_{n\in\mathbb{N}} \E[\,\lVert D F_n\lVert_{\cH^2}]<\infty.\] Then $F$ belongs to $\mathbb{D}^{1,2}$, and the sequence of derivatives $\{D F_n,\,n\in\IN\}$
%converges to $D F$ in the weak topology of
%$L^2(\Omega\times [0,T])$.
%\end{lemma}

\subsection{Some results on BMO martingales}\label{bmo-subsection}
BMO martingales play a key role for a priori estimates needed in our
sensitivity analysis of solutions of BSDE. For details about their
theory we refer the reader to \cite{Kazamaki1994}.

If $\Phi$ is a square integrable martingale with $\Phi_0=0$, the martingale representation theorem yields a square integrable process $\phi$ such that $\Phi_t=\int_0^t \phi_s \udws,
t\in[0,T]$. Hence the $BMO(\IQ)$ norm can be alternatively expressed as
\[\sup_{\tau\in[0,T]} \E^\IQ\Big[ \int_\tau^T \phi_s^2 \uds|\cF_\tau \Big]< \infty.\]
As an easy consequence, if $\Phi \in BMO$ then $\int H \ud \Phi\in BMO$ for any bounded adapted process $H$.

\begin{lemma}[Properties of BMO martingales]\label{bmoeigen}
Let $\Phi$ be a BMO martingale. Then we have:
\begin{itemize}
  \item[1)] The stochastic exponential $\cE(\Phi)$ is uniformly
  integrable.
  \item[2)] There exists a number $r>1$ such that $\cE(\Phi_T)\in L^r$. This property follows from the \emph{Reverse H\"older inequality}.
  The maximal $r$ with this property can be expressed explicitly in terms of the BMO norm of $\Phi$. There exists as well an upper bound for $\| \cE(\Phi_T) \|_{L^r}^r$ depending only on $T$, $r$ and the BMO norm of $\Phi$.
  \item[3)] For probability measures $\IP$ and $\IQ$ satisfying $ \ud \IQ=\cE(\Phi_T)\ud \IP$, and for $\Phi\in BMO(\IP)$, the process $\hat{\Phi}= \Phi - \langle \Phi\rangle$ is a $BMO(\IQ)$
  martingale.
  \item[4)] Energy  inequalities imply the inclusion $BMO \subset \cH^p$ for all $p\geq 1$. More precisely, for $\Phi=\int_0^\cdot \phi_s \uds\in BMO$
  with BMO norm $C$, and $p\geq 1$ the following estimate holds
\begin{align*}%\label{hp-made-easy}
\E[\Big(\int_0^T |\phi_s|^2\uds\Big)^p ]\leq 2p!(4C^2)^p.
\end{align*}
\end{itemize}
\end{lemma}

\subsection{The setting and its assumptions}

For functions $b$, $\sigma$, $g$ and $f$, for $x\in \IR^m$ and a
$d$-dimensional Brownian motion $W$ we intend to study the solution
processes of the following system of forward-backward stochastic
differential equations (with generators of quadratic growth
(qgFBSDE)). For $t\in[0,T]$ they are given by
\begin{align}
\label{fbsde-sde}
X_t&=x+\int_0^t b(s,X_s)\uds+\int_0^t \sigma(s,X_s)\udws,\\
\label{fbsde-bsde}
Y_t &= \xi- \int_t^T Z_s \udws + \int_t^T f(s, \Theta_s) \uds,
\end{align} with $\xi=g(X_T)$ and $\Theta_s=(X_s, Y_s, Z_s)$.

For the functions figuring in the above system of equations we
hierarchically order the properties they will be assumed to satisfy.
\bit
\item[{\bf HX0}] There is a constant $K$ such that $b,\sigma_i:[0,T]\times\IR^m\to\IR^m, 1\le i\le d,$ are uniformly Lipschitz continuous with Lipschitz constant $K$, and $b(\cdot,0)$ and $\sigma_i(\cdot,0), 1\le i\le d,$ are bounded
by $K$.

\item[{\bf HX1}] Hypothesis {\bf HX0} holds. For any $0\le t\le T$ the functions $b(t,\cdot),\sigma_i(t, \cdot), 1\le i\le d,$ are differentiable and its derivatives are uniformly Lipschitz with Lipschitz constant $K$ independent of $t$.
In other words, $\sigma\in B_m^{m\times d}$ and $b\in B_m^{m\times
1}$. There exists a positive constant $c$ such that
\begin{align}\label{unif-ellipt}
y^T \sigma(t,x)\sigma^T(t,x)y \geq c|y|^2,\quad x,y\in\IR^m,\ t\in[0,T].
\end{align}

\item[{\bf HX2}] Hypothesis {\bf HX1} holds. There exists a positive constant $K$ such that $b(t,\cdot)\in C^2_b(\R^m)$
and $\sigma(t, \cdot)\in C^2_b(\R^{m\times d})$ with second
derivatives bounded by $K$.

\item[{\bf HY0}] There is a positive constant $M$ such that $g:\R^m\to \R$ is absolutely uniformly bounded by $M$, hence $|\xi|\leq M$. $f:[0,T]\times\R^m\times\R\times\R^d\to \R$
is an adapted measurable function, continuous in the space variables, for which there exists a positive constant $M$ such that for all $t\in[0,T]$, $x,x'\in\IR^m$, $y,y'\in\IR$ and $z,z'\in\IR^d$
\begin{align*}
|f(t,x,y,z)|&\leq M(1+|y|+|z|^2),\\
|f(t,x,y,z)-f(t,x,y',z')|&\leq M|y-y'|+M(1+|z|+|z'|)|z-z'|\\%.
|f(t,x,y,z)-f(t,x',y,z)|&\leq M(1+|y|+|z|^2)|x-x'|
\end{align*}

\item[{\bf HY1}] Hypothesis {\bf HY0} holds. $f$ is differentiable in $(x,y,z)$ and there exists $M\in\IR_+$ such that
\begin{align*}
|\nabla_x f(t,x,y,z)|&\leq M(1+|y|+|z|^2),\\
|\nabla_y f(t,x,y,z)|&\leq M,\\
|\nabla_z f(t,x,y,z)|&\leq M(1+|z|).
\end{align*}
$g:\R^m\to\R$ is a Lipschitz differentiable function satisfying
$|\nabla g|\leq M$.

\item[{\bf HY2}] Hypothesis {\bf HY1} holds, $g \in C^2_b(\R^m)$. The driver $f$ is twice differentiable with continuous second order derivatives. There exists an
adapted process $(K_t)_{0\leq t\leq T}$ belonging to $\cS^{2p}(\IR)$ for all $p\geq 1$ such that for any $t\in[0,T]$ all second order derivatives of $f$
at $(t,\Theta_t)=(t,X_t,Y_t,Z_t)$ are a.s.  dominated by $K_t$.
\eit

\subsection{Some results on SDE}\label{subsec:sde-results}
We recall the results on SDE known from the literature that are
relevant for this work. We state our assumptions in the
multidimensional setting. However, for ease of notation we present
some formulas in the one dimensional case.\footnote{For a beautiful
presentation of this subsection's Theorems we point the reader to
\cite{phd-elie}.}
\begin{theo}[Moment estimates for SDE]
Assume that {\bf HX0} holds. Then
(\ref{fbsde-sde}) has a unique solution and the following moment estimates hold: for any $p\geq  2$ there exists a constant $C>0$, depending only on $T$, $K$ and $p$ such that for any $x\in\R^m, s,t\in[0,T]$
\begin{align}
\label{supnorm-X}
\E[\,\sup_{0\leq t\leq T} |X_t|^p\,]&\leq C\E\Big[\, |x|^p+\int_0^T \big(|b(t,0)|^p+|\sigma(t,0)|^p \big)\ud t \Big],\\
\label{delta-X}
\E[\sup_{s\leq u\leq t}|X_u-X_s|^p\,]&\leq C \E\Big[\,|x|^p+\sup_{0\leq t\leq T}\big\{|b(t,0)|^p+|\sigma(t,0)|^p\big\}  \Big]\,|t-s|^{p/2}.
\end{align}
Furthermore, given two different initial conditions $x,x'\in\IR^m$
and denoting the respective solutions of (\ref{fbsde-sde}) by $X^x$
and $X^{x'}$, we have
\begin{align*}%\label{compare-sdex-sdey}
\E\Big[\sup_{0\leq t\leq T}| X_t^x-X_t^{x'} |^p\Big]\leq C |x-x'|^p.
\end{align*}
\end{theo}

\begin{theo}[Classical differentiability]\label{theo.SDE-classic-diff}
Assume {\bf HX1} holds. Then the solution process $X$ of (\ref{fbsde-sde}) as a function of the initial condition $x\in\IR^m$ is differentiable and
satisfies for $t\in[0,T]$
\begin{align}\label{diff-sde}
\nabla X_t &= I_m +\int_0^t \nabla b(s, X_s) \nabla X_s \uds +\int_0^t
\nabla \sigma(s, X_s)\nabla X_s\udws,
\end{align}
where $I_m$ denotes the $m\times m$ unit matrix. Moreover, $\nabla
X_t$ as an $m\times m$-matrix is invertible for any $t\in[0,T]$. Its
inverse $(\nabla X_t)^{-1}$ satisfies an SDE and for any $p\geq 2$
there are positive constants $C_p$ and $c_p$ such that
\begin{align}\label{supnorm-gradX}
\| \nabla X \|_{\cS^p}+\| (\nabla X)^{-1}\|_{\cS^p}&\leq C_p
\end{align}
and
\begin{align}\label{delta-nablaX-1}
\E\Big[\sup_{s\leq u\leq t}|(\nabla X_u)-(\nabla X_s)|^p+\sup_{s\leq u\leq t}|(\nabla X_u)^{-1}-(\nabla X_s)^{-1}|^p\,\Big]&\leq c_p \,|t-s|^{p/2}.
\end{align}
\end{theo}

\begin{theo}[Malliavin Differentiability]\label{X-malliavin-diff}
Under {\bf HX1}, $X\in \mathbb{L}_{1,2}$ and its Malliavin derivative admits a version $(u,t)\mapsto D_u X_t$ satisfying for $0\le u\le t\le T$ the SDE
\begin{align*}
D_u X_t &= \sigma(u, X_{u})+\int_u^t \nabla b(s, X_s) D_u X_s \uds
+\int_u^t \nabla \sigma(s, X_s) D_u X_s\udws.
\end{align*}
Moreover, for any $p\geq 2$ there is a constant $C_p>0$ such that for $x\in\R^m$ and $0\leq v \leq u\leq t\leq s\leq T$
\begin{align*}
\| D_u X\|_{\cS^p}^p&\leq C_p(1+|x|^p),\\
\E[\, |D_u X_t - D_u X_s|^p ]&\leq  C_p(1+|x|^p)|t-s|^{\frac{p}{2}},\\
\| D_u X-D_v X\|_{\cS^p}^p&\leq  C_p(1+|x|^p)|u-v|^{\frac{p}{2}}.
\end{align*}
By Theorem \ref{theo.SDE-classic-diff}, we have the representation
\[D_u X_t= \nabla X_t (\nabla X_u)^{-1}\sigma(u, X_u)\1_{[0,u]}(t),\quad \textrm{ for all }u,t\in[0,T].\]

If {\bf HX2} holds, then $D X\in \IL_{1,2}$. For all $v,u,t\in[0,T]$, $D_vD_u X_t$ admits a version which solves for $0\leq v\leq u\leq t\leq T$
\begin{align*}
D_v D_u X_t& = \nabla \sigma (u, X_{u})D_v X_u+\nabla \sigma (v, X_{v})D_u X_v\\
&\qquad+\int_u^t\Big[ \nabla b(s, X_s)D_v D_u X_s+ \Delta b(s, X_s) D_v X_s D_u X_s \Big]\uds\\
&\qquad+\int_u^t\Big[ \nabla \sigma(s, X_s)D_v D_u X_s+ \Delta \sigma (s, X_s)D_v X_s D_u X_s  \Big]\udws.
\end{align*}
Furthermore, there exists a continuous version of $(D_v D_u X_t)_{v,u,t\in[0,T]}$ such that for all $0\leq v,u\leq T$ and $p\geq 2$ we have
\[\| D_u D_v X\|_{\cS^p}^p\leq C_p(1+|x|^{2p}).\]
For $0\leq v,v'\leq u,u'\leq t\leq T$ we have
\[\| D_v D_u X-D_{v'}D_{u'} X\|_{\cS^{p}}^{p}\leq C_p(|v-v'|^{\frac{p}{2}}+|u-u'|^{\frac{p}{2}}).\]
\end{theo}

\subsection{Results on BSDE with drivers of quadratic growth}\label{subsec:qgBSDE}

We next collect some results on qgBSDE. For their original versions
or more information, we refer to  \cite{00Kob}, \cite{AIDR07},
\cite{07BriCon} and \cite{phd-dosReis}.

\begin{theo}[Properties of qgBSDE]\label{theo:moment-estimates-special-class}
Under {\bf HY0}, {\bf HX0}, the system
(\ref{fbsde-sde}), (\ref{fbsde-bsde}) has a unique solution
$(X,Y,Z)\in \cS^2 \times \cS^\infty \times \cH^2$. The norms of $Y$
and $Z$ depend only on $T$, $K$, $M$ as given by assumption {\bf
HY0}.

The martingale $Z*W$ belongs to the space of BMO martingales, and
hence $Z\in\cH^p$ for all $p\geq 2$. The following estimate
holds\footnote{This inequality follows from applying It\^o's formula
to $e^{a Y_t+bt}$ with an appropriate choice of $a$ and $b$.}:
\begin{align*}
\| Z*W \|_{BMO}\leq \frac{4+6M^2T}{3M^2} \exp \Big\{6M\|\xi\|_{L^\infty}+MT\Big\}<\infty.
\end{align*}
\end{theo}
\begin{remark} \label{def:r-bar}
Following point \emph{2)} of Lemma \ref{bmoeigen}, we define a pair
$(\barr,\barq)$ such that $1/\barr+1/\barq=1$ and $\cE(Z*W)\in
L^{\barr}$.

In the following, when discussing BMO martingales, an appearing
exponent $\barr$ will always be used in this sense.
\end{remark}
For more properties about BMO martingales in the setting of BSDE with drivers of quadratic growth we refer to Lemma 2.1 in \cite{AIDR07}.

The two differentiability results we now present can be found in
\cite{phd-dosReis}. These results are natural extensions of results
proved in \cite{AIDR07} or \cite{07BriCon}. For further details,
comments and complete proofs we refer to \cite{phd-dosReis}.
\begin{theo}[Classical differentiability]\label{theo:1st-order-diff-qgbsde}
Suppose that {\bf HX1} and {\bf HY1} hold. Then for all $p\geq 2$ the solution processes $(X^x,Y^x,Z^x)$ of the system (\ref{fbsde-sde}), (\ref{fbsde-bsde}) with initial vector
$x\in\R^m$ for the forward component belongs to $\cS^p\times \cS^p\times \cH^p$.
The application $\IR^m\ni x\mapsto (X^x,Y^x,Z^x)\in \cS^p(\IR^m)\times \cS^p(\IR)\times \cH^p(\IR^d)$ is differentiable. The derivatives of $X$ satisfy (\ref{diff-sde}) while the derivatives of $(Y,Z)$ satisfy the linear BSDE
\begin{align}
\label{diff-bsde}
\nabla Y_t^{x}&= \nabla g(X_T^{x})\nabla X^x_T-\int_t^T \nabla Z_s^{x}\ud W_s+\int_t^T
\langle \nabla f(s,\Theta^x_s) , \nabla \Theta^x_s \rangle \ud s.
\end{align}
If ${\bf HX2}$ and ${\bf HY2}$ hold, then there exists a version of the solution $\Omega\times [0,T]\times \IR^m\ni (\omega,t,x)\mapsto (X^x_t,Y^x_t,Z^x_t)(\omega)\in \IR^{m}\times\IR^{1}\times\IR^{d}$,
such that for almost all $\omega$, $X^x$ and $Y^x$ are continuous in time and continuously differentiable in $x$.
\end{theo}

\begin{theo}[Malliavin differentiability]\label{theo:bsde-1d-mall-diff}
Suppose that {\bf HX1} and {\bf HY1} hold. Then the solution processes $(X,Y,Z)$ of system (\ref{fbsde-sde}), (\ref{fbsde-bsde}) verify
\begin{itemize}
\item
for any $0\leq t\leq T$, $x\in\R^m$ we have $(Y_t,Z_t)\in
\mathbb{L}_{1,2}\times\big(\mathbb{L}_{1,2}\big)^d$. X satisfies the
statement of Theorem \ref{X-malliavin-diff}, and a version of $(D_u
Y_t,D_u Z_t)_{0\leq u,t\leq T}$ satisfies
\begin{align}\label{aux:1order-mall-dif-bsde}
D_u Y_t &= 0, \qquad D_u Z_t = 0,\qquad t<u\le T,\nonumber\\
D_u Y_t &= \nabla g(X_T) D_u X_T + \int_t^T \langle \nabla f(s,\Theta_s), D_u \Theta_s \rangle \ud s - \int_t^T  D_u Z_s \ud W_s, \qquad t\in
[u,T].
\end{align}
Moreover, $(D_t Y_t)_{0\le t
\leq T}$ defined by the above equation is a version of $(Z_t)_{0\le t\leq T}$.
\item the following representation holds for any $0\leq  u\leq t\leq T$ and $x\in\R^m$
\begin{align}
\nonumber%\label{DY-rep}
D_u Y_t &= \nabla_x Y_t (\nabla_x X_u)^{-1}\sigma(u,X_u),\quad a.s.,\\
\label{z-rep}
Z_t&=\nabla_x Y_t (\nabla_x X_t)^{-1}\sigma(s,X_t),\quad a.s..
\end{align}
\end{itemize}
\end{theo}

\section{Inequalities for BSDE with stochastic Lipschitz conditions}\label{sec:bmo-BSDE}

In this section we look closely at BSDE with drivers that satisfy
Lipschitz conditions with random Lipschitz constants. Our interest
in this problem is motivated by the following observation. If we
formally differentiate the driver of our original BSDE, we see that
the essential term $Z^2$ produces a term of the form $Z DZ$. In this
term we may consider the factor $Z$ as a random growth rate of the
factor $DZ$.

Let $\zeta$ be a random variable and $f$ a measurable function. We
consider the BSDE
\begin{equation} \label{bsde-lin-001}
U_t = \zeta - \int_t^T V_s \ud W_s + \int_t^T f(\cdot, s,U_s,V_s) \ud s,\qquad t\in[0,T].
\end{equation}
We state a set of assumptions for $\zeta$ and $f$. For $p\geq 1$ we
stipulate
\begin{itemize}%\label{hypothesis:A12345}
\item[(HA1)] $\zeta$ is $\cF_T$-adapted random variable and $\zeta \in L^{2p}(\R)$.
\item[(HA2)] $f:\Omega\times [0,T]\times \R \times \R^d\to \R$ is product measurable and  there exists a positive constant $M$ and a positive predictable process $H$ such that for all $u,u'\in\R$ and $v,v'\in\R^d$ we have
\begin{align*}%\label{assump:A2}
|f(\cdot,\cdot,u,v)-f(\cdot,\cdot,u',v')|\leq M|u-u'|+H_{\cdot}|v-v'|,
\end{align*} and such that $H*W$ is a $BMO$ martingale.
\item[(HA3)] $\big(f(\cdot,t,0,0)\big)_{t\in [0,T]}$ is a measurable $(\cF_t)$-adapted process satisfying such that for all $p\geq 1$ we have $\E[\big(\int_0^T |f(\cdot,s,0,0)|\ud s\big)^{p}] < \infty$.% for all $p \ge 1$.
\end{itemize}
Moreover, we assume that $(U,V)$ is a solution of BSDE (\ref{bsde-lin-001}), and the constant $\barr$ is related to the BMO martingale $H*W$ as in Remark \ref{def:r-bar}.

\subsection{Moment estimates for BSDE with random Lipschitz constant}

For the study of sensitivity properties of solutions of qgFBSDE, as
seen in \cite{AIDR07} or \cite{07BriCon}, it is convenient to
consider BSDE with random Lipschitz constants. The moment estimates
for this type of BSDE one finds in the two cited papers still
leave space for improvements. A weakness of the results of
\cite{AIDR07}, owed to the techniques used, is the lack of an
estimate for $\|U\|_{\cS^2}$. We next state an extended moment
estimate, obtained by using ideas of \cite{07BriCon}.
\begin{lemma}\label{apriori.estimate}
Let (HA1) through (HA3) be satisfied and take $p \geq 1$. Let $\barr>1$ be such that
$\mathcal{E}(H*W) \in L^{\barr}(\IP)$. Then there exists a positive constant $C$, depending only on $p$, $T$, $M$ and the BMO-norm of $H*W$,
such that with the conjugate exponent $\barq$ of $\barr$ we have
\begin{align}\label{basic-moment.estimate}
\|U\|_{\cS^{2p}}^{2p}+\|V\|_{\cH^{2p}}^{2p}
\leq C\E\Big[ |\zeta|^{2p\barq^2}+\Big(\int_0^T |f(\cdot,s,0,0)|\uds\Big)^{2p\barq^2} \Big]^{\frac{1}{\barq^2}}.
\end{align}
\end{lemma}
\begin{proof}
Assumption (HA2) states that the driver is Lipschitz continuous in
$u$. We first use this hypothesis to simplify the BSDE. For $t\in[0,T]$ we define
\[
a_t=\frac{f(\cdot,t,U_t,V_t)-f(\cdot,t,0,V_t)}{U_t}\, 1_{\{U_t\not= 0\}} \quad \textrm{ and }\quad
e_t=\exp\big\{ \int_0^t a_s \uds \big\}.
\]
Under (HA2), namely the Lipschitz property of $f$ in the first spatial variable,
the process $a$ is well defined and absolutely uniformly bounded by $M$. Hence $e$ is
bounded from above and from below by a positive constant. For $t\in[0,T]$ we further define
\[
b_t=\frac{f(\cdot,t,0,V_t)-f(\cdot,t,0,0)}{|V_t|^2}V_t \, 1_{\{V_t\not= 0\}}.
\]
By (HA2), $b$ is well defined and bounded in absolute value by the process $H$. Applying It\^o's formula to $(e_t U_t)_{t\in[0,T]}$ we obtain
\begin{align*}
e_t U_t = e_T \zeta +\int_t^T e_s\big[f(\cdot,s,0,0)+b_sV_s\big]\uds -\I e_s V_s \udws.
\end{align*}
We simplify the BSDE further by defining a new measure $\IQ^b$ for
which $W^b=W-\int_0^\cdot b_s \uds$ is a $\IQ^b$-Brownian motion. The
Radon-Nikodym density of $\IQ^b$ with respect to $\IP$ is given by
the stochastic exponential $\cE(b*W)$. Since $|b|\leq H$ we have
$\|b*W\|_{BMO}\leq \|H*W\|_{BMO}$. Hence the measure $\IQ^b$ is
indeed a probability measure. For $t\in[0,T]$ our BSDE takes the
form
\begin{align}\label{loc:linearizedBSDE}
e_t U_t = e_T \zeta +\int_t^T e_s f(\cdot,s,0,0)\uds -\I e_s V_s \ud
W_s^{b}.
\end{align} We now proceed with moment estimates. Taking conditional expectations with respect to $\IQ^b$, estimating by absolute values and integrating on the whole interval we obtain
\begin{align*}
|e_t U_t| \leq \E^{\IQ^b}\Big[e_T |\zeta| +\int_0^T e_s|f(\cdot,s,0,0)|\uds\, \big|\cF_t\Big].
\end{align*}
Applying Doob's moment inequality for $2p\geq 2$ we obtain a similar
inequality as in this Theorem's statement, but under the measure
$\IQ^b$, i.e.
\begin{align}\label{loc:ineq-U-under-Q}
\| U \|_{\cS^{2p}(\IQ^b)}^{2p} \leq C \E^{\IQ^b}\Big[|\zeta|^{2p} +\Big(\int_0^T |f(\cdot,s,0,0)|\uds\Big)^{2p}\Big].
\end{align} If we rewrite equation (\ref{loc:linearizedBSDE}), isolate the stochastic integral on the left hand side and take $t=0$, use
Burkholder-Davis-Gundy's inequality and remember that $e$ is also
bounded from below by a positive constant thanks to (HA2), we get
\begin{align*}
\E^{\IQ^b}\Big[\Big( \int_0^T |V_s|^2 \uds\Big)^p\Big] &\leq c_p
\E^{\IQ^b}\Big[  |e_T \zeta|^{2p}+ \sup_{0\leq t\leq T} |e_t U_t|^{2p}  +\Big(\int_0^T |e_s f(\cdot,s,0,0)|\uds \Big)^{2p} \Big]\\
&\leq C \E^{\IQ^b}\Big[|\zeta|^{2p} +\Big(\int_0^T
|f(\cdot,s,0,0)|\uds\Big)^{2p}\Big].
\end{align*} For the second inequality we used (\ref{loc:ineq-U-under-Q}) and the fact that $[\int_0^T |V_s|^2\uds]^{1/2}$ is integrable.

Summing the last two inequalities we get
\begin{align}\label{temp-Q-bounds}
\|U\|_{\cS^{2p}(\IQ^b)}^{2p}+\|V\|_{\cH^{2p}(\IQ^b)}^{2p}
\leq C\E^{\IQ^b}\Big[ |\zeta|^{2p}+\Big(\int_0^T |f(\cdot,s,0,0)|\uds\Big)^{2p} \Big].
\end{align}

This inequality is already close to the one we have to deduce. To
complete the proof, we just have to get rid of the dependence on $\IQ^b$ in the terms of the inequality. We do this for
(\ref{loc:ineq-U-under-Q}), noting that for the other inequality the
arguments are very similar. As mentioned before, $b$ is dominated by
$H$ and therefore \[ \|b*W\|_{BMO}\leq\|H*W\|_{BMO}.\] Further, part \emph{3)} of Lemma \ref{bmoeigen} implies that since $b*W\in BMO(\IP)$,
also $(-b)*W^{b}\in BMO(\IQ^{b})$. Moreover, since
$[\cE(b*W)]^{-1}=\cE\big((-b)*W^{b}\big)$, part 2) of the same Lemma
states the existence\footnote{Here we follow the notation we
stipulated in Remark \ref{def:r-bar}.} of a real number $\barr > 1$
for which $\cE(b * W) \in L^\barr(\IP)$ and $[\cE(b*W)]^{-1} \in
L^\barr(\IQ^b)$. The constant $\barr$ is estimated from the
$BMO(\IP)$ norm of $H*W$, as indicated in Lemma \ref{bmoeigen}.

Throughout let $D = \max\big\{ \| \cE(b*W) \|_{L^\barr(\IP)},
\|\cE(b*W)^{-1}\|_{L^\barr(\IQ^b)} \big\}$ and let $\barq$ be the
conjugate H\"older exponent of $\barr$.

Combining (\ref{loc:ineq-U-under-Q}) and H\"older's inequality, we
obtain for any $p\geq 1$
\begin{align*}
\E^{\IP}[\sup_{s\in[0,T]} |U_s|^{2p}]
&= \E^{\IQ^b}\big[\cE(b*W)^{-1} \sup_{s\in[0,T]} |U_s|^{2p}\big]\
\leq\ D \E^{\IQ^b}\big[ \sup_{s\in[0,T]} |U_s|^{2p\barq}\big]^{\frac{1}{\barq}} \\
&\leq C_1\,D\,\E^{\IQ^b}\Big[ |\zeta|^{2p\barq} + \Big(\int_0^T |f(\cdot,s,0,0)| \ud s\Big)^{2p\barq} \Big]^{\frac{1}{\barq}}\\
&=C_1\,D\, \E^{\IP}\Big[\cE(b*W) \Big( |\zeta|^{2p\barq} + \Big(\int_0^T |f(\cdot,s,0,0)| \ud s \Big)^{2p\barq}\Big)\Big]^{\frac{1}{\barq}}\\
&\leq C_2 \,D^{\frac{1+\barq}{\barq}} \E^{\IP}\Big[|\zeta|^{2p\barq^2} + \Big(\int_0^T |f(\cdot,s,0,0)| \ud s \Big)^{2p\barq^2}\Big]^{\frac{1}{\barq^2}},
\end{align*}
where $C_1, C_2$ represent constants depending on $p, M, T$. % and the $BMO$ norm of $(b*W)$.
Similarly, with another constant $C_3$,
\[\E^{\IP}\Big[\Big(\int_0^T |V_s|^{2} \ud s\Big)^p\Big] \leq C_3\,D^{\frac{1+\barq}{\barq}} \E^{\IP}\Big[|\zeta|^{2p\barq^2} + \Big(\int_0^T |f(\cdot,s,0,0)| \ud s \Big)^{2p\barq^2}\Big)\Big]^{\frac{1}{\barq^2}}.\]
Combining the two estimates we obtain (\ref{basic-moment.estimate}).
\end{proof}

\subsection{A priori estimates for BSDE with random Lipschitz constant}%\label{secpriori}

In this section, following the results of the previous one, we
derive a priori inequalities which serve in the usual way to compare
solutions of BSDE of the type considered obtained for different
system parameters such as initial states of the forward part. This
result will later be used to determine the good candidates for the
derivatives of our original qgBSDE.

For each $i\in\{1,2\}$, let $\zeta_i$ be a random variable
satisfying condition (HA1) and $f_i$ a driver function satisfying
(HA2) and (HA3) with respective square integrable processes $H^i$ such
that $H^i*W\in BMO$. With this random variable and driver function
we investigate the following BSDE
\begin{equation} \label{bsde.compar}
U_t^{(i)} = \zeta_i - \int_t^T V_s^{(i)} \ud W_s + \int_t^T
f_i(\omega, s, U_s^{(i)}, V_s^{(i)}) \uds,\quad t\in[0,T].
\end{equation}
\begin{lemma}\label{comparison.theo}
Assume the conditions of Lemma \ref{apriori.estimate} hold for
(\ref{bsde.compar}). Take further $\barq$ with respect to the BSDE
with $i=1$. Then we have for any $p\geq 1$ a positive constant $C$ exists such that
\begin{align*}
&\|U^{(1)}-U^{(2)}\|_{\cS^{2p}}^{2p}+\| V^{(1)}-V^{(2)} \|_{\cH^{2p}}^{2p}\\
&\qquad \leq C \,\E\Big[ |\zeta_1-\zeta_2|^{2p\barq^2}+
\Big(\int_0^T | (f_1-f_2)(\cdot,s,U_s^{(2)}, V_s^{(2)}) | \uds
\Big)^{2p\barq^2} \Big]^{\frac{1}{\barq^2}},
\end{align*} with $\barq$ given as in Remark \ref{def:r-bar} with respect to the BMO martingale $(H^1*W)$.
\end{lemma}
\begin{proof}
The arguments to prove this inequality are similar to those used in
the proof of Lemma \ref{apriori.estimate}. Therefore we will omit
some of the already familiar details.

Define $\delta U=U^{(1)}-U^{(2)}$, $\delta V=V^{(1)}-V^{(2)}$,
$\delta \zeta=\zeta_1-\zeta_2$ and $\delta
f(\cdot,t,u,v)=(f_1-f_2)(\cdot,t,u,v)$. Then to simplify the BSDE define $a$ and
$b$ for $t\in[0,T]$ by
\begin{align*}
a_t&=\frac{f_1(\cdot,s,U^{(1)}_t,V^{(1)}_t)-f_1(\cdot,s,U^{(2)}_t,V^{(1)}_t)}{U_t^{(1)}-U_t^{(2)}} \1_{\{U_t^{(1)}\neq U_t^{(2)}\}},\\
b_t&=\frac{f_1(\cdot,s,U^{(2)}_t,V^{(1)}_t)-f_1(\cdot,s,U^{(2)}_t,V^{(2)}_t)}{|V_t^{(1)}-V_t^{(2)}|^2}(V_t^{(1)}-V_t^{(2)})\1_{\{V_t^{(1)}\neq V_t^{(2)}\}}.
\end{align*}
We arrive at an equation similar to
(\ref{loc:linearizedBSDE}) given by:
\begin{align*}
e_t \delta U_t = e_T \delta \zeta+\int_t^T [e_s\, \delta f(\cdot,s,
U^{(2)}_s, V^{(2)}_s)]\,\uds -\int_t^T \delta V_s \ud W_s^{b},
\end{align*}
with $W^b = W - \int_0^\cdot b_s \uds.$  Define $\IQ^b$ with respect
to $b$ as before. Now we may proceed as in the proof of Lemma
\ref{apriori.estimate}. The existence of the integral of $\delta
f(\cdot,s, U^{(2)}_s, V^{(2)}_s)$ is justified by observing that we can
dominate $\delta f$ using our assumptions and also because Lemma
\ref{apriori.estimate} is applicable to each individual BSDE. The
result follows.
\end{proof}
Without prior knowledge of the form of $f_1$ and $f_2$ the right
hand side of the Lemma's inequality cannot be treated further. In
the following result we assume that the drivers satisfy a stochastic
linearity property. Then the increment in the drivers can be further
estimated.
\begin{coro}
Assume the conditions of Lemma \ref{comparison.theo} are
satisfied, and furthermore that for each $i\in\{1,2\}$ the driver $f_i$ is
linear, i.e. it satisfies
\[
f_i(\cdot, t, u,v)= \alpha_i(\cdot,t)+\beta_i(\cdot, t)u+\langle \gamma_i(\cdot,t), v\rangle ,
\]
with $(\alpha_i,\beta_i,\gamma_i)$ adapted random processes
belonging to $\cH^{2p}(\IR)\times \cS^\infty(\IR)\times
\cH^2(\IR^d)$ for any $p\geq 1$. Moreover, we assume that $\beta_i$
is bounded and that $(\gamma_i*W)\in BMO$. Then
\begin{align*}
&\|U^{(1)}-U^{(2)}\|_{\cS^{2p}}^{2p}+\| V^{(1)}-V^{(2)} \|_{\cH^{2p}}^{2p}\\
&\quad \leq C\Big\{\E\Big[
|\zeta_1-\zeta_2|^{2p\barq^2}+
\Big(\int_0^T | \alpha_1(\cdot,s)-\alpha_2(\cdot,s)| \uds\Big)^{2p\barq^2}\Big]^{\frac{1}{\barq^2}}\\
&\qquad\quad+\E\Big[\Big(\int_0^T | \beta_1(\cdot,s)-\beta_2(\cdot,s)| \uds\Big)^{4p\barq^2}
+ \Big(\int_0^T | (\gamma_1(\cdot,s)-\gamma_2(\cdot,s)|^2 \uds\Big)^{2p\barq^2}
\Big]^{\frac{1}{2\barq^2}}\Big\}
\end{align*}
\end{coro}
\begin{proof}
Starting with the inequality of Lemma \ref{comparison.theo}, and
injecting the new assumptions, we obtain
\begin{align*}
&\|U^{(1)}-U^{(2)}\|_{\cS^{2p}}^{2p}+\| V^{(1)}-V^{(2)} \|_{\cH^{2p}}^{2p}\\
&\hspace{1.5cm}\leq C\E\Big[
|\zeta_1-\zeta_2|^{2p\barq^2}+
\Big(\int_0^T | \alpha_1(\cdot,s)-\alpha_2(\cdot,s)| \uds\Big)^{2p\barq^2}\\
&\hspace{3cm}+\Big(\sup_{0\leq t\leq T} |U^{(2)}_t|\Big)^{2p\barq^2} \Big(\int_0^T | \beta_1(\cdot,s)-\beta_2(\cdot,s)| \uds\Big)^{2p\barq^2}\\
&\hspace{4.5cm}+ \Big(\int_0^T |V^{(2)}_s|^2 \uds\Big)^{p\barq^2}\Big(\int_0^T | (\gamma_1(\cdot,s)-\gamma_2(\cdot,s)|^2 \uds\Big)^{p\barq^2}
\Big]^{\frac{1}{\barq^2}}.
\end{align*}
The moment estimates of Lemma \ref{apriori.estimate} ensure that
$\|U^{(2)}\|_{\cS^{4p\barq^2}}$ and  $\|V^{(2)}\|_{\cH^{4p\barq^2}}$
are finite. Hence a simple application of H\"older's inequality
yields the desired result.
\end{proof}

\section{Second order Malliavin differentiability}\label{section:3}

We now give sufficient conditions on our system of stochastic equations which ensure the solution processes are twice Malliavin differentiable.

\subsection{The main result}

\begin{theo}\label{theo:2-order-mall}
Assume {\bf HX2} and {\bf HY2} hold. Then the solution process
$\Theta = (X, Y, Z)$ of the qgFBSDE
(\ref{fbsde-sde}), (\ref{fbsde-bsde}) is twice Malliavin
differentiable, i.e. for each $u\in[0, T]$ and $i\in\{1,\ldots,d\}$
the processes $(D_u^iY_t, D^i_u Z_t) \in
\IL_{1,2}\times(\IL_{1,2})^d$. A version of $\{(D^j_v D_u^i Y_t,D^j_v D^i_u Z_t);0\leq v\leq u\leq t\leq T\}$ with $0\leq
j,i\leq d$ satisfies
\begin{align}\label{2order-mall-bsde}
D^j_v D^i_u Y_t &= D^j_v D^i_u \xi -\int_t^T D^j_v D^i_u Z_s\ud W_s\\
&\nonumber\qquad
+\int_t^T \Big[ (D^j_v \Theta_s)^T \big[Hf\big](s,\Theta_s) D^i_u \Theta_s + \langle \nabla f(s,\Theta_s),D^j_vD^i_u \Theta_s\rangle\Big]\ud s,
\end{align}
where $\big[Hf\big]$ is the Hessian matrix of the function $f$ and
$\xi=g(X_T)$. Considered as a BSDE, (\ref{2order-mall-bsde}) admits
a unique solution.

Moreover $\{D_t D_u Y_t; 0\leq u \leq t\leq T\}$ is a version of
$\{D_u Z_t; 0\leq u \leq t\leq T\}$.
\end{theo}

By Theorem \ref{X-malliavin-diff}, condition {\bf HX2} already
implies that $X\in \IL_{2,p}$. Therefore one only needs to prove the
Malliavin differentiability of $(DY, DZ)$.

\subsection{Strategy of the proof}

The main problem in proving the variational differentiability of
equation (\ref{aux:1order-mall-dif-bsde}) is given by the growth of
$\nabla_z f(\cdot,z)$ in $z$. {\bf HY1} states that $\nabla_z
f(\cdot, z)$ is dominated by $C(1+|z|)$. Considering
(\ref{aux:1order-mall-dif-bsde}) as a BSDE with solution process
$(DY, DZ)$ leads to interpreting the influence of $\nabla_z
f(\cdot,z)$ in the driver as a random Lipschitz constant. We aim at
using the same strategy of proof as in \cite{AIDR07}: we approximate
the BSDE (\ref{aux:1order-mall-dif-bsde}) by truncating the random
Lipschitz constant, and then use Lemma \ref{lemma.from.nualart.extended} to
obtain variational differentiability in the limit. Therefore we
mainly have to establish the conditions of Lemma
\ref{lemma.from.nualart.extended}.

\subsubsection{A differentiable truncation family for the identity function}
%
%
%
%
%By defining
%\begin{align*}
%f_n:\Omega\times [0,T]\times \IR \times \IR^d &\to \IR,\\
%(t,y,z)&\mapsto f_n(t,y,z)=f(t,y,h_n(z)).
%\end{align*}
%where $f$ is the driver of qgBSDE (\ref{linbsde.aux015}), we obtain the following family of BSDE:
%
%%%%%%%%%%%%%%%%%%%%%%%%%%
%
%********************
%
We start by introducing a sequence of smooth real valued functions $(\tilde{h}_n)_{n\in\IN}$ that truncate the identity on the real line and that will be used to truncate the variable $z$ in the function $\nabla_z f(\cdot,\cdot,\cdot,z)$. We choose $\tilde{h}_n:\IR\to\IR$ continuously differentiable with the following properties:
\bit
\item $(\tilde{h}_n)_{n\in\IN}$ converges locally uniformly to the identity; For all $n\in\IN$ and  $z\in\IR$ it holds that $|\tilde{h}_n(z)|\leq |z|$, $|\tilde{h}_n(z)|\leq n+1$ and
\begin{equation}\label{trunc-family}
\tilde{h}_n(z)= \left\{
\begin{array}{cl}
(n+1)&,z> n+2,\\
z&,|z|\leq n,\\
-(n+1)&,z<-(n+2).
\end{array}
\right.
\end{equation}
\item The derivative of $\tilde{h}_n$ is absolutely bounded by $1$, and converges to $1$ locally uniformly.
\eit
We remark that such sequence of functions exists. The above requirements are for instance consistent with
\begin{displaymath}
\tilde{h}_n(z)= \left\{
\begin{array}{cl}
\big(-n^2+2nz-z(z-4)\big)/4&,z\in[n,n+2],\\
\big(n^2+2nz+z(z+4)\big)/4&,z\in[-(n+2),-n].\\
\end{array}
\right.
\end{displaymath}
We then define $h_n: \IR^d\to \IR^d$ by $z\mapsto
h_n(z)=(\tilde{h}_n(z_1),\cdots, \tilde{h}_n(z_d))$, $n\in\IN$.

%%%%%%%%%%%%%%%%%%%%%%%%%%
%So we introduce a sequence of  functions $(h_n)_{n\in\IN}$ that truncate the identity function on the real line, and that will be used to
%truncate the $z$ variable in $\nabla f(\cdot,z)$ in the driver term of (\ref{aux:1order-mall-dif-bsde}).  We choose $h_n: \IR\to \IR$ continuously differentiable with the following properties:
%
%*****************
%
%We remark that such sequence of functions exists. The above
%requirements are for instance consistent with
%\begin{displaymath}
%h_n(z)= \left\{
%\begin{array}{cl}
%\big(-n^2+2nz-z(z-4)\big)/4&,z\in[n,n+2],\\
%\big(n^2+2nz+z(z+4)\big)/4&,z\in[-(n+2),-n].\\
%\end{array}
%\right.
%\end{displaymath}
%

\subsubsection{The family of truncated FBSDE and results concerning them}

Recall the notation $\Theta=(X,Y,Z)$ for the solution of system
(\ref{fbsde-sde}), (\ref{fbsde-bsde}), the driver of BSDE
(\ref{aux:1order-mall-dif-bsde}) with terminal condition $\xi =
g(X_T)$, where $g$ is a bounded differentiable function and {\bf
HX1} is satisfied. For $n\in\IN$ take the sequence $(h_n)_{n\in\IN}$
defined in (\ref{trunc-family}) and define the sequence of
approximate drivers $F^n:\Omega\times[0,T]\times \IR^m\times
\IR\times\IR^d\to \IR$ by
\begin{align}\label{2D-mall-diff-truncated-driver}
F^n(t,x,u,v)&=\big\langle\nabla_x f(t,\Theta_t),x\big\rangle + \nabla_y f(t,\Theta_t)u
 +\big\langle\nabla_z f\big(t,X_t,Y_t,h_n(Z_t)\big),v\big\rangle.
\end{align}
The advantage of approximating the driver in this way is a technical
one: we can make use of the well known $\Theta$ and its properties,
and do not have to deal with approximations of $\Theta$ and its
Malliavin derivatives at the same time.

For $i\in\{1,\cdots,d\}$, $0\leq u\leq t\leq T$ and $n\in\IN$ consider the following BSDE
\begin{align}\label{1dbsde-sequence}
U^n_{u,t} &= D^i_u \xi +
\int_t^T F^n(s, \Xi^n_{u,s})\ud s- \int_t^T V^n_{u,s} \ud W_s, \quad \Xi_{u,s}^n=(D^i_u X_s, U^n_{u,s}, V^n_{u,s}),
\end{align}
where $D^i \xi$, $D^i X$ denote the first Malliavin derivatives of
$\xi$ and $X$ respectively.

In the following Lemma we state existence, uniqueness and Malliavin
differentiability of the solution processes of BSDE
(\ref{1dbsde-sequence}). The Lemma's proof will result from a
theorem formulated in the Appendix, where all the hypotheses,
variants of the hypotheses employed in Theorem
\ref{theo:2-order-mall}, are formulated. To avoid repetitions, we do
not formulate them here again.
\begin{lemma}[2nd order Malliavin diff. of Lipschitz BSDE]\label{Lip-2nd-mall-diff}
For each $n\in\IN$, (\ref{1dbsde-sequence}) has a unique
solution $(U^n,V^n)$ in
$\cS^{2p}([0,T]\times[0,T])\times\cH^{2p}([0,T]\times[0,T])$ for any $p\geq 1$.

Furthermore for $0\leq u\leq t\leq T$ the random variables
$(U^n_{u,t},V^n_{u,t})$ are Malliavin differentiable and for any
$j\in\{1,\cdots,d\}$ a version of $\{(D^j_v U^n_{u,t},D^j_v
V^n_{u,t});0\leq v\leq u \leq t\leq T\}$ satisfies
\begin{align}\label{2dbsde}
D^j_v U^n_{u,t} &= D^j_vD^i_u \xi -\int_t^T D^j_v V^n_{u,s}\ud W_s\nonumber \\
&\qquad+\int_t^T \Big[
(D^j_v F^n)(s,\Xi^n_{u,s})+\langle (\nabla F^n)(s,\Xi^n_{u,s}) , D^j_v \Xi^n_{u,s}\rangle
\Big]\ud s,
\end{align}
with $\Xi_{u,s}^n=(D^i_u X_s, U^n_{u,s}, V^n_{u,s})$ and $D^j_v
\Xi^n_{u,s}=(D^j_vD^i_u X_s,D^j_v U^n_{u,s},D^j_v V^n_{u,s}), 0\leq
v\leq u \leq s\leq T$.
\end{lemma}
For clarity of  exposition we give a few words about the driver of
BSDE (\ref{2dbsde}). Assuming $d=m=1$ and hence omitting the
superscripts $i$ and $j$, and denoting $\Theta=(X,Y,Z)$ and
$\Theta^n=(X,Y,h_n(Z))$, we can describe the first term inside the
integral by
\begin{align*}
(D_v F^n)(t,\Xi^n_{u,t}) &= D_v[(\nabla_x f)(t,\Theta_t)]D_uX_t +D_v[(\nabla_y f)(t,\Theta_t)]U^n_{u,t}+D_v[(\nabla_z f)(t,\Theta^n_t)]V_{u,t}^n.
\end{align*} These three terms can be further specified by
\begin{align*}
&D_v[(\nabla_x f)(t,\Theta_t)]D_uX_t\\
&\qquad= (\nabla_{xx}f)(t,\Theta_t)D_v X_t D_uX_t
+(\nabla_{xy}f)(t,\Theta_t)D_v Y_t D_uX_t+(\nabla_{xz}f)(t,\Theta_t)D_v Z_t D_uX_t,
\end{align*} an analogous expression for $D_v[(\nabla_y f)]$, while the last part is given by
\begin{align*}
&D_v[(\nabla_z f)(t,\Theta^n_t)]V^n_{u,t}\\
&\qquad= (\nabla_{zx}f)(t,\Theta^n_t)D_v X_t V^n_{u,t}
+(\nabla_{zy}f)(t,\Theta^n_t)D_v Y_t
V^n_{u,t}+(\nabla_{zz}f)\big(t,\Theta^n_t)h'_n(Z_t)D_v Z_t
V^n_{u,t}.
\end{align*}
The second term of the driver in (\ref{2dbsde}) can be expressed by
\begin{align*}
&\langle (\nabla F^n)(s,\Xi^n_{u,s}) , D_v \Xi^n_{u,s}\rangle \\
&\qquad= \nabla_x f(s,\Theta_s) D_v D_u X_s + \nabla_y f(s,\Theta_s)
D_v U^n_{u,s} +\nabla_z f\big(s,X_s,Y_s,h_n(Z_s)\big) D_v V^n_{u,s}.
\end{align*}

To compact notation a bit, we denote the driver component in
(\ref{2dbsde}) not containing $D_v U^n_{u,s}$ and $D_v V^n_{u,s}$ by
\begin{align}\label{aux:An-proc}
A^n_{v,u,s}&= (D_v F^n)(s,\Xi^n_{u,s})+(\nabla_x
F^n)(s,\Xi^n_{u,s})D_v D_u X_s,\quad v,u,s\in[0,T].
\end{align}

Before giving the proofs of Theorem \ref{theo:2-order-mall} or Lemma \ref{Lip-2nd-mall-diff} we
prove two helpful Lemmas.

\begin{remark}\label{barr.indep.n}
Since $|h_n(z)|$ is dominated by $|z|$, it is clear from {\bf HY1} that
\[
\sup_{n\in\IN}
\| \nabla_z f\big(t,X,Y,h_n(Z)\big)*W \|_{BMO} \leq C \| \big(1 + |Z|\big)*W \|_{BMO}<\infty.\]
Hence by Lemma \ref{bmoeigen}, there exists a $\barr$ such that the stochastic exponentials related to the two BMO martingales above belong both to $L^\barr$. We remark that $\barr$
is independent of $n$.
\end{remark}

\begin{lemma}\label{aux:bounds-DY-DZ} Assume {\bf HX2} and {\bf HY2}
hold, that $(U^n,V^n)$ solve BSDE (\ref{1dbsde-sequence}) and
$(DY,DZ)$ BSDE (\ref{aux:1order-mall-dif-bsde}). Then, for any
$p\geq 1$ we have
\begin{align*}
\sup_{n\in\IN} \sup_{0\leq u\leq T}\left\{ \IE\Big[ \Big(\int_0^T
|D_u Y_s|^2+|D_u Z_s|^2\ud s\Big)^{p}
 +\Big(\int_0^T |U^n_{u,s}|^2+|V^n_{u,s}|^2\ud s\Big)^{p}\Big]
\right\} < \infty.
\end{align*}
\end{lemma}
\begin{proof}
For any $n\in\IN, z\in\IR$ our hypothesis gives $|h_n(z)|\le |z|$.
Hence the driver $F^n$ of (\ref{2D-mall-diff-truncated-driver})
satisfies the same growth conditions as the driver of BSDE
(\ref{aux:1order-mall-dif-bsde}). Therefore, one can apply the
results of Section \ref{sec:bmo-BSDE} to either BSDE and obtain for
$p\ge 1$
\begin{align*} &
\sup_{0\leq u\leq T}\Big\{
\IE
\Big[\Big(\int_0^T|D_u Y_s|^{2}\uds\Big)^{p}+\Big( \int_0^T|D_u Z_s|^2\ud s\Big)^p + \Big(\int_0^T|U^n_{u,s}|^{2}\uds\Big)^p+\Big( \int_0^T|V^n_{u,s}|^2\ud s\Big)^p  \Big]\Big\}\\
&\qquad\qquad\qquad\quad\leq
    C \sup_{0\leq u\leq T}\IE\Big[\Big(|D_u \xi|^{2} +
    \int_0^T |\nabla_x f(s,\Theta_s)D_u X_s|^2\ud s\Big)^{p\barq^2}\Big]^{\frac{1}{\barq^2}},
\end{align*}
with $\barq$ the H\"older conjugate of $\barr$. The results of subsection
\ref{subsec:sde-results} combined with assumptions {\bf HX2} and
{\bf HY2} yield the finiteness of the right hand side of the
inequality.
\end{proof}

\begin{lemma}\label{aux:lemma-finite-RHS}
Assume {\bf HX2} and {\bf HY2} hold. For all $p\geq 1$ we have
\begin{align*}
\sup_{n\in\IN} \sup_{0\leq u,v \leq T} \IE\Big[\, |D_vD_u\xi|^{2p} +
\Big(\int_0^T |A^n_{v,u,s}| \ud s\Big)^{2p} \Big]  <\infty,
\end{align*} with $A^n, n\in\IN,$ given by (\ref{aux:An-proc}).
\end{lemma}
\begin{proof}
To prove this result we analyze each term in more detail.

\emph{Part 1):} The first term presents little difficulty, since
$\xi=g(X_T)$ and $X$ is a diffusion process. For $0\le v\le u\le T$
we have $D_v D_u\xi = D_v X_T [Hg](X_T) D_u X_T + \nabla g(X_T)D_v
D_u X_T,$ where $[Hg]$ is the Hessian matrix of $g$.

Since $g \in C^2_b$, we may use the inequality $2ab\leq (a^2+b^2)$
valid for $a, b\in\IR$ combined with Theorem \ref{X-malliavin-diff}
to obtain
\begin{align*}
\sup_{0\leq u,v\leq T} \IE\Big[\, |D_vD_u\xi|^{2p}\Big]
&\leq
C\sup_{0\leq u,v\leq T} \IE\Big[\,
%|D_v X_T|^{4p}+
|D_u X_T|^{4p}+|D_vD_uX_T|^{2p}\Big]<\infty.
\end{align*}

\emph{Part 2):} We now analyze the second term, starting with the
identification
\begin{align*}
A^n_{v,u,s}&=(D_v F^n)(s,\Xi^n_{u,s})+(\nabla_x F^n)(s,\Xi^n_{u,s})D_v D_u X_s\\
&=(D_v F^n)(s,\Xi^n_{u,s})+(\nabla_x f)(s,\Theta_{s})D_v D_u X_s,
\end{align*}
for $v,u,s\in[0,T].$ Now $(D_v F^n)(t,\Xi^n_{u,s})$ is composed of
products of first order Malliavin derivatives of $X,Y$ or $Z$ and
second order partial derivatives of $f$. Assumption {\bf HY2}
guarantees that the second order derivatives of $f$ are dominated by
a process $K$ belonging to $\cS^{2p}([0,T])$. Combining this with
the hypothesis $|h'_n|\leq 1$ for all $n$ we easily obtain
\begin{align*}
|(D_v F^n)(s,\Xi^n_{u,s})|\leq CK_s\Big\{|D_v X_s|^2+|D_u
X_s|^2+|D_v Y_s|^2+|U^n_{u,s}|^2+|D_v Z_s|^2+|V^n_{u,s}|^2\Big\}.
\end{align*}
Summands involving the Malliavin derivatives of $X$ can be dealt
with arguments as in \emph{part 1)} of this proof. Furthermore,
\begin{align*}
&\sup_{n\in\IN} \sup_{0\leq u\leq T} \IE\Big[ \Big(\sup_{0\leq t\leq T}|K_t|\int_0^T \Big[\,|U^n_{u,s}|^2+|V^n_{u,s}|^2+|D_u Y_s|^2+|D_u Z_s|^2\Big]\ud s\Big)^{2p}\Big]\\
& \hspace{4cm} \leq \|K \|_{\cS^{4p}}^{2p} \sup_{n\in\IN}
\sup_{0\leq u\leq T} \| \,|U^n_{u}|+|V^n_{u}|+|D_u Y|+|D_u Z|\,
\|_{\cH^{8p}}^{4p} <\infty.
\end{align*} The last inequality is satisfied by Lemma \ref{aux:bounds-DY-DZ} and the fact that $K\in\cS^{2p}$ for all $p\geq 1$.

We are left with the analysis of the term $(\nabla_x
f)(s,\Theta_{s})D_v D_u X_s$. From condition {\bf HY1},
$\nabla_xf(s,\Theta_s)$ is dominated by $M(1+|Y_s|+|Z_s|^2)$ with a
bounded process $Y$, and so we obtain
\begin{align*}
&\sup_{0\leq u,v \leq T}\IE\Big[\, \Big(\int_0^T |(\nabla_x f)(t,\Theta_{s})D_v D_u X_s| \ud s\Big)^{2p} \Big]\\
&\qquad\leq C \sup_{0\leq u,v \leq T}\IE\Big[ \sup_{t\in[0,T]}|D_v D_u X_t|^{2p} \Big( \int_0^T (1+|Z_s|^2) \ud s\Big)^{2p} \Big]\\
&\qquad\qquad\leq C \sup_{0\leq u,v \leq T}\big\| D_vD_uX \big\|_{\cS^{4p}}^{2p}\big\| 1+|Z| \big\|_{\cH^{8p}}^{4p}<\infty
\end{align*}
For the last two inequalities we used H\"older's inequality, that
$Z\in\cH^{2p}$ for all $p\geq 1$ and Theorem \ref{X-malliavin-diff}.

The Lemma's inequality follows from a combination of \emph{Part 1)} and \emph{Part 2)}.
\end{proof}
We are now in a position to prove Lemma \ref{Lip-2nd-mall-diff}. We will use Theorem  \ref{appendixB:2-order-mall-diff-lip-BSDE} stated in the appendix.
\begin{proof}[Proof of Lemma \ref{Lip-2nd-mall-diff}]
We have to establish the hypotheses to hold for the application of
Theorem \ref{appendixB:2-order-mall-diff-lip-BSDE}. The terminal
condition is given by the Malliavin derivative of $\xi=g(X_T)$ with
$g\in C_b^2$. In view of Theorems \ref{theo.SDE-classic-diff} and
\ref{X-malliavin-diff}, conditions (A2) and (A4) are satisfied.

Given our construction, it is clear that for each fixed $n\in\IN$, the driver $F^n$ is uniformly Lipschitz continuous in $(y,z)$, since $\nabla_y f$ and $\nabla_z f(\cdot,h_n(\cdot))$ are bounded. The boundedness of $\nabla_x f$ combined with the fact that $DX\in
\cS^{2p}([0,T]\times[0,T])$ enables us to conclude \[\sup_{0\leq
u\leq T} \E[\,\sup_{0\leq t\leq T}|\nabla_x
f(t,\Theta_t)D_uX_t|^{2p}]<\infty,\]
and hence condition (A1) is also
satisfied.

The verification of condition (A3) is also simple. $F^n$ is
continuous differentiable in $(y,z)$. Furthermore since $Y$ and $Z$
are Malliavin differentiable  and $X$ is twice Malliavin
differentiable, we have that $F^n(t, D_u X_t, 0, 0)$, $F^n(t, 0, 1, 0)$ and $F^n(t, 0, 0, 1)$ are also Malliavin
differentiable for $0\le u\le t\le T$. The proof of the moment
inequality of assumption (A3) is a consequence of Lemma
\ref{aux:lemma-finite-RHS}.

Hence we may apply Theorem \ref{appendixB:2-order-mall-diff-lip-BSDE}.
\end{proof}

\subsection{Proof of Theorem \ref{theo:2-order-mall}}
We are finally able to prove the main result of this section.

\begin{proof}[Proof of Theorem \ref{theo:2-order-mall}]
To prove this result we apply Lemma \ref{lemma.from.nualart.extended}. We
have to show that the Lemma's assumptions are satisfied. Fix $0\le
u\le t\le T.$

1) Lemma \ref{Lip-2nd-mall-diff} ensures existence, uniqueness and
Malliavin differentiability of each $(U_{u,t}^n, V_{u,t}^n)$.

2) We now prove the $\cH^2$-convergence of
$(U^n_{u,\cdot},V^n_{u,\cdot})$ to $(D_uY_\cdot,D_u Z_\cdot)$. Using
Lemma \ref{comparison.theo} applied to the BSDE resulting from the
difference $D_uY_\cdot-U^n_{u,\cdot}$ (see BSDE
(\ref{aux:1order-mall-dif-bsde}) and (\ref{1dbsde-sequence})), we
have with $\Theta=(X,Y,Z)$ and $\Theta^n=(X,Y,h_n(Z))$
\begin{align*}
&\sup_{0\leq u\leq T}
\IE\Big[ \int_0^T | D_u Y_s-U^n_{u,s}|^{2}\ud s+\int_0^T |D_uZ_s-V^n_{u,s}|^2\ud s\Big]\\
&\quad\leq C \sup_{0\leq u\leq T} \E\Big[\Big(\int_0^T |\nabla_z f (s,\Theta_s)-\nabla_z f(s,\Theta^n_s)| |V^n_{u,s}| \uds\Big)^{2\barq^2}\Big]^{\frac{1}{\barq^2}}\\
&\quad\leq  C \sup_{0\leq u\leq T}
\E\Big[\Big(\int_0^T |V^n_{u,s}|^2\uds\Big)^{2\barq^2}\Big]^{\frac{1}{2\barq^2}}  \E\Big[\Big(\int_0^T |\nabla_z f (s,\Theta_s)-\nabla_z f(s,\Theta^n_s)|^2\uds\Big)^{2\barq^2}\Big]^{\frac{1}{2\barq^2}},
\end{align*}
where $\barq$ is related to the BMO martingale $(\nabla_z
f(\cdot,X,Y,Z))*W$ as stated in subsection \ref{bmo-subsection}.

The first term in the last line is finite, uniformly in $n$, by
Lemma \ref{aux:bounds-DY-DZ}. For the second term, note that by {\bf
HY1} $\nabla_z f$ is continuous and, from (\ref{trunc-family}) so is
the family $h_n$. Furthermore, both $\nabla_z f(\cdot,\cdot,z)$ and
$\nabla_z f(\cdot,\cdot,h_n(z))$ are dominated by $C(1+|z|)$. Given
the integrability properties of $Z$ and the convergence of $h_n$ to
the identity function, dominated convergence yields the desired
convergence result, from which the convergence of
$(U^n_{u,t},V^n_{u,t})$ to $(D_uY_t,D_u Z_t)$ for a.e. $t\in[0,T]$
follows.

3) We prove the uniform boundedness of $\| (D U^n,D
V^n)\|^2_{L^2(\ID^{1,2})}$ in $n$.

The driver of BSDE (\ref{2dbsde}) is linear. So applying Lemma
\ref{apriori.estimate}, we obtain the following inequality for
$n\in\IN$
\begin{align*}
\|D_v U^n_u\|_{\cS^{2p}}^{2p}+\|D_v V^n_u\|_{\cH^{2p}}^{2p}\leq
C\E\Big[ |D_vD_u\xi|^{2p\barq^2}+\Big(\int_0^T
|A^n_{v,u,s}|\uds\Big)^{2p\barq^2} \Big]^{\frac{1}{\barq^2}},
\end{align*} where the constant $\barq$ is related to the BMO martingale $\nabla_z f(\cdot,\Theta)*W$ according to Remark \ref{def:r-bar}.
Lemma \ref{aux:lemma-finite-RHS} now yields
\[
\sup_{n\in\IN} \sup_{0\leq v,u\leq T} \Big\{\|D_v
U^n_u\|_{\cS^{2p}}^{2p}+\|D_v V^n_u\|_{\cH^{2p}}^{2p}\Big\}<\infty.
\]

By 1) to 3) we can apply Lemma \ref{lemma.from.nualart.extended} and
deduce the Malliavin differentiability of $(D Y, D Z)$. Arguments as
the ones used in Theorem 8.4 of \cite{AIDR07} show that $(D_v D_u Y,
D_v D_u Z)$ is a solution to BSDE (\ref{2order-mall-bsde}).

Uniqueness follows immediately from Lemma \ref{comparison.theo}.

To prove the representation $D_tD_uY_t=D_uZ_t$, one only needs to
recall that for each $n, u\le t$ we have $D_tD_uY_t^n=D_uZ_t^n$.
Since both sides converge to their respective limiting processes the
equality holds true in the limit.
\end{proof}

\section{Regularity in the time variable}\label{sec:reg-in-time}

With a view towards their numerical approximation, in this section we investigate regularity properties of the Malliavin derivatives of solutions of our qgFBSDE (\ref{fbsde-sde}), (\ref{fbsde-bsde}).

In the following subsections the results presented are shown to hold mainly under assumptions {\bf HX1} and {\bf HY1}.
Several of these results can be proved under weaker conditions, namely by replacing {\bf HY1} with {\bf HY0}.
This is achieved by using a canonical argument of regularization followed by the application of Fatou's lemma.
Because this type of reasoning is well known we state only Theorem \ref{zhangs-path-reg-theo-to-qgfbsde} under weaker assumptions.

\subsection{Continuity and bounds}\label{subsec:cont-bounds}

\begin{lemma}\label{kolm-for-DY-DZ}
Under {\bf HX1} and {\bf HY1} let $(X,Y,Z)$ be the solution processes of system (\ref{fbsde-sde}), (\ref{fbsde-bsde}), and $(DX,DY,DZ)$ their Malliavin derivatives.
Then for $u,v\in[0,T]$ and $p\geq 1$ there exists a positive constant $C_p$ such that
\begin{align*}
\|D_vY-D_u Y\|_{\cS^{2p}}^{2p}+\|D_vZ-D_u Z\|_{\cH^{2p}}^{2p} \leq
C_p |v-u|^{p}.
\end{align*}
\end{lemma}
\begin{proof}
We use (\ref{aux:1order-mall-dif-bsde}) to write for $u,v,t\in[0,T]$
a FBSDE for the difference $D_v Y_t-D_u Y_t$. For this we employ the
comparison Lemma \ref{comparison.theo}, to obtain with $\xi=g(X_T)$
and for any $p\geq 1$
\begin{align*}
&\|D_vY-D_u Y\|_{\cS^{2p}}^{2p}+\|D_vZ-D_u Z\|_{\cH^{2p}}^{2p}\\
&\leq C \E\Big[\,|D_v \xi-D_u \xi|^{2p\barq^2}
+\Big(\int_0^T |(\nabla_x f)(s,\Theta_s)||D_v X_s - D_u X_s| \uds\Big)^{2p\barq^2}\Big]^{\frac{1}{\barq^2}}\\
&\leq C\E\Big[\,|D_v X_T-D_u X_T|^{2p\barq^2}\\
&\qquad \qquad +\sup_{0\leq t\leq T}|D_v X_t - D_u X_t|^{2p\barq^2} \Big(  \int_0^T (1+|Y_s|+|Z_s|^2) \uds\Big)^{2p\barq^2}\Big]^{\frac{1}{\barq^2}}\\
&\leq C \Big\{ \|D_v X_T -D_uX_T \|_{L^{2p\barq^2}}^{2p}+\|D_v X
-D_uX \|_{\cS^{4p\barq^2}}^{2p} \Big\} \leq C_p |v-u|^{p},
\end{align*}
where $\barq$ corresponds to the BMO martingale
$\nabla_zf(\Theta)*W$. The last line follows from a direct
application of Theorem \ref{X-malliavin-diff}.
\end{proof}

Equipped with these moment estimates we are now able to state our
first main result.

\begin{theo}[Time continuity]\label{DY-continuity}
Assume {\bf HX1} and {\bf HY1}. Then there exists a continuous
version of $(u,t) \mapsto D_uY_t$ in $\{(u,t): 0\le u\le t\le T.\}$
In particular there exists a continuous version of $Z$ on $[0,T].$

Assume {\bf HX2} and {\bf HY2}. Then there exists a continuous
version of $(v,u,t)\mapsto D_vD_uY_t$ for $0\leq v\leq u\leq t\leq
T$. In particular there is a continuous version of $(u,t)\mapsto D_u
Z_t$ for $0\leq u\leq t\leq T$.
\end{theo}
\begin{proof}
To make the proof simpler we assume $m=d=1$. Under {\bf HX1}, the
results of subsection \ref{subsec:sde-results} imply the existence
of continuous versions of $X$, $\nabla X$, $(\nabla X)^{-1}$ and
$(u,t)\mapsto D_uX_t$ for $0\le u\le t\le T$.

A quick analysis of (\ref{diff-bsde}), combined with the knowledge
that $(X,Y,Z)\in \cS^{2p} \times \cS^{\infty} \times \cH^{2p}$ and
$(\nabla X, \nabla Y, \nabla Z)\in \cS^{2p} \times \cS^{2p} \times
\cH^{2p}$ for all $p\geq 1$, allows one to conclude that a
continuous version of $\nabla Y$ exists: the process is given by the
sum of a Lebesgue and It\^o integral with well behaved integrands.

In Theorem \ref{theo:bsde-1d-mall-diff} we established $D_u Y_t =
\nabla Y_t (\nabla X_u)^{-1}\sigma(X_u)$, $0\le u\le t\le T$.
Condition {\bf HX0} ensures the continuity of $\sigma$. Given that all terms in the representation of $D_uY_t$ are continuous, we conclude that there is a continuous version of $(u,t) \mapsto
\big(\nabla Y_t (\nabla X_u)^{-1}\sigma(X_u)\big)$ for $0\leq u\leq
t\leq T$. This means that $(u,t)\mapsto D_u Y_t$ has a continuous
version for $0\leq u\leq t\leq T.$

By Theorem \ref{theo:bsde-1d-mall-diff}, $Z$ is a version of
$t\mapsto D_tY_t$. Hence the continuity of a version of
$(u,t)\mapsto D_u Y_t$ for $0\le u\le t\le T$ immediately implies
that $Z$ possesses a continuous version. This finishes the proof of
the first statement.

For the second statement we argue in a different way.

The second Malliavin derivative of $Y$ depends on three variables, $v,u,t\in[0,T]$. By using moment inequalities, we will show that $(v,u,t)\mapsto D_v D_u Y_{t}$ is continuous as a mapping to the space of continuous functions on $0\leq v\leq u\leq t\leq T$ equipped with the $\sup$ norm. By well known extensions of the Kolmogorov continuity criterion to normed vector spaces (see for example Theorem 1.4.1 in \cite{Kunita1990}) this will establish the desired continuity of $(v,u,t)\mapsto D_v D_u Y_t$ for $0\le v\le u\le t\le T.$  To verify the inequalities, for $0\le v\le u\le T$ and $0\le v'\le u'\le T$ we will have to estimate moments of
\[
\sup_{0\le t\le T} |D_{v} D_{u} Y_t - D_{v'} D_{u'} Y_t|^p.
\]
In a first step, we separate the two parameters by estimating this quantity by a constant multiple of
$$\sup_{0\le t\le T} |D_v D_u Y_t - D_{v'} D_u Y_t|^p + \sup_{0\le t\le T} |D_{v'} D_u Y_t - D_{v'}
D_{u'} Y_t|^p.$$ In what follows, for convenience we shall only give
the estimation of the first summand, remarking that the second one
may be treated in a very similar way. Fix $0\leq v,v'\leq u\leq t\leq T$. Again using the
comparison Lemma \ref{comparison.theo} with (\ref{2order-mall-bsde})
specified to $D_vD_uY_t-D_{v'}D_uY_t$, we get for $p\geq 1$
\begin{align*}
&\E\Big[\sup_{0\leq t\leq T}|D_vD_uY_t-D_{v'}D_u Y_t|^{2p}\Big]\\
&\quad\leq C \Big\{\E\Big[\,|D_vD_u \xi-D_{v'}D_u \xi|^{2p\barq^2}+\Big(\int_0^T \Big[\, |D_v \Theta_s-D_{v'} \Theta_s||[Hf](s,\Theta_s)||D_u \Theta_s|\\
&\qquad\qquad \qquad+|(\nabla_x f)(s,\Theta_s)||D_v D_uX_s - D_{v'}D_u X_s|\,\Big] \uds\Big)^{2p\barq^2}\Big]\Big\}\\
&\quad\leq C \Big\{
\| D_v X_T - D_{v'}X_T \|_{L^{6p\barq^2}}^{2p}
+\|D_v D_u X - D_{v'}D_u X\|_{\cS^{4p\barq^2}}^{2p}\\
&\qquad\qquad + \|D_v X - D_{v'}X\|_{\cS^{4p\barq^2}}^{2p}+ \|D_v Y - D_{v'}Y\|_{\cS^{4p\barq^2}}^{2p}+ \|D_v Z - D_{v'}Z\|_{\cH^{8p\barq^2}}^{4p}
\Big\}\\
&\quad \leq C_p |v-v'|^{p}.
\end{align*}
The successive inequalities are justified in view of the growth conditions
contained in the assumptions, H\"older's inequality, Theorem
\ref{X-malliavin-diff} and Lemma \ref{kolm-for-DY-DZ}. Kolmogorov's
continuity criterion for vector valued stochastic processes yields
the existence of a continuous version of $(v,u,t)\mapsto D_vD_uY_t$
for $0\leq v\leq u\leq t\leq T$, and hence by restriction also of
$(u,t)\mapsto D_u Z_t$ for $0\le u\le t\le T.$
\end{proof}
\begin{theo}[Bounds]\label{Z-sinfty-bounds}
Assume that {\bf HX1} and {\bf HY1} hold. Then for all $p\geq 1$
\begin{align*}%\label{DY-sup}
\E\Big[ \sup_{0\leq u\leq t\leq T} |D_uY_t|^{2p} \Big]&<\infty.
\end{align*}
In particular
\begin{align}
\label{Z-moment-time} \lVert Z \lVert_{\cS^{2p}}\ &<\infty.
\end{align}
Let {\bf HX2} and {\bf HY2} be satisfied. Then for all $p\geq1$
\begin{align*}%\label{supDZ}
\sup_{0\leq u\leq T}\IE[\,\sup_{0\leq t\leq T} |D_uZ_t|^{2p}\,]<\infty.
\end{align*}
\end{theo}
\begin{proof}
As we have seen in Theorem \ref{DY-continuity}, a continuous version
of $(u,t)\mapsto D_uY_t$ is given by $\nabla_x Y_t (\nabla_x
X_u)^{-1}\sigma(u,X_u)$. Hence we may estimate
\begin{align*}
\E\Big[ \sup_{0\leq u\leq t\leq T} |D_uY_t|^{2p} \Big]
&\leq \E\Big[ \sup_{0\leq t\leq T} |\nabla_x Y_t|^{2p} \sup_{0\leq u\leq T} \big\{|(\nabla_x X_u)^{-1}\sigma(u,X_u)|^{2p} \big\}\Big]\\
& \leq \E\Big[ \sup_{0\leq t\leq T} |\nabla_x Y_t|^{6p}\Big]^{\frac13}
\E\Big[ \sup_{0\leq u\leq T} |(\nabla_x X_u)^{-1}|^{6p}\Big]^{\frac13}
\E\Big[ \sup_{0\leq u\leq T} |\sigma (X_u)|^{6p}\Big]^{\frac13}\\
&<\infty
\end{align*}
The last line follows from the fact that $\nabla Y$, $(\nabla
X)^{-1}$ and $X$ all belong to $\cS^{2p}$ for all $p\geq 1$ (see
(\ref{supnorm-X}), (\ref{supnorm-gradX}) and Theorem
\ref{theo:1st-order-diff-qgbsde}). This concludes the first part of
the proof. The second claim follows as a special case of the first
by identifying $u$ and $t$.

For the third statement, note that the proof of Theorem
\ref{theo:2-order-mall} (see also the proof of
\ref{appendixB:2-order-mall-diff-lip-BSDE}) yields
\[\sup_{0\leq v,u\leq T} \E\Big[ \sup_{0\leq t\leq T} |D_vD_u Y_t|^{2p} \Big]<\infty, \quad p\geq 1.  \]
By the continuity result of Theorem \ref{DY-continuity} we may
choose $u=t$ to obtain
\[\sup_{0\leq u\leq T} \E\Big[ \sup_{0\leq t\leq T} |D_u Z_t|^{2p} \Big]<\infty, \quad p\geq 1.  \]
\end{proof}

\subsection{A path regularity theorem}

In the previous subsection we deduced the continuity property of $Z$
and estimated moments of its supremum over the interval $[0,T]$. Here,
we aim at providing a Kolmogorov continuity type estimate for $Z$.
The inequality we will obtain will imply an improvement of the well
known path regularity result stated in \cite{phd-zhang} and \cite{Zhang2004}.

Let $\Pi$ be the collection of all partitions of the interval $[0,T]$ by finite families of real numbers. Particular partitions will be denoted by $\pi=\{t_i: 0=t_0<\ldots<t_N=T\}$ with $N\in\IN$.
We define the mesh size of partition $\pi$ as
$\Delta^\pi=\Delta=\max_{0\leq i\leq N}|t_{i+1}-t_i|$.

For reference purposes and before approaching the path regularity
theorem we recall an elementary inequality: for real numbers $a_i,
1\le i\le n,$ and $p\ge 1$ we have
\begin{align}
\label{ineq:discrete-newton} \sum_{i=1}^n |a_i|^p & \leq  \Big( \sum_{i=1}^n |a_i| \Big)^p.
\end{align} We start by stating an auxiliary lemma.
\begin{lemma} Assume {\bf HX0} and {\bf HY0}. Then for the solutions of BSDE (\ref{fbsde-sde}), (\ref{fbsde-bsde}) and for any $p\geq 2$ there exists a pair of constants $A_p,\,C_p$ depending on $T$, $M$ and $p$ such that
\begin{align}\label{delta-t-Y}
\E[\sup_{s \leq u\leq t} |Y_u-Y_s|^p \,]&\leq
C_p \Big\{ A_p |t-s|^p+\E\Big[ \Big(\int_s^t|Z_v|^2\ud v\Big)^{p}+ \Big(\int_s^t|Z_v|^2 \ud v\Big)^{p/2}\Big]\Big\}.
\end{align}
\end{lemma}
\begin{proof}
First estimate increments of $Y$ by the sum of a Lebesgue and It\^o
integral provided by (\ref{fbsde-bsde}), maximize in $s\le u\le t$,
and apply Doob's and Burkholder-Davis-Gundy's inequalities to the
martingale part to obtain for $p\geq 2$
\begin{align*}
\E[\sup_{s\leq u\leq t} |Y_u-Y_s|^p\, ]\leq C_p \E\Big[ \Big(\int_s^t |f(v, X_v, Y_v, Z_v)|\ud v\Big)^p + \Big(\int_s^t |Z_v|^2 \ud v\Big)^{\frac{p}{2}}\Big].
\end{align*}
Next use the growth condition valid for $f$, i.e.
$|f(\cdot,\cdot,y,z)|\leq M(1+|y|+|z|^2)$ together with the fact
that $Y$ is bounded, to obtain the claimed result.
\end{proof}

Let us now state our path regularity theorem.
\begin{theo}[Path regularity]\label{theo:path-reg} Under {\bf HX1} and {\bf HY1}, the FBSDE system (\ref{fbsde-sde}), (\ref{fbsde-bsde}) has a unique solution $(X,Y,Z)\in\cS^{2p}\times\cS^\infty\times\cH^{2p}$ for
all $p\geq 1$. Moreover, the following holds true:
\bit
\item[i)]
For $p\geq 2$ there exists a constant $C_p>0$ such that for $0\leq
s\leq t\leq T$
 we have
\[\E[\sup_{s\leq u\leq t} |Y_u-Y_s|^p \,]\leq C_p |t-s|^{\frac{p}{2}}.\]

\item[ii)]
For all $p\geq 1$ there exists a constant $C_p>0$ such that for any
partition $\pi$ of $[0,T]$ with mesh size $\Delta$
\[
\sum_{i=0}^{N-1} \E\Big[\Big(\int_{t_i}^{t_{i+1}}|Z_t-Z_{t_i}|^2\udt\Big)^p\Big]\leq C_p \Delta^p.
\]
\eit
Under {\bf HX2} and {\bf HY2}, we further have:
\bit
\item[iii)]
For all $p\geq 2$ there exists a constant $C_p>0$ such that for
$0\leq s\leq t\leq T$

\[\E[\sup_{s\leq u\leq t} |Z_u-Z_s|^p \,]\leq C_p |t-s|^{\frac{p}{2}}.\]
In particular, the process $Z$ has a continuous modification.
\eit
\end{theo}
\begin{proof}

\emph{Part i)}: Under the hypotheses we can make use of
Theorem \ref{Z-sinfty-bounds}. In fact, combining
(\ref{Z-moment-time}) with (\ref{delta-t-Y}) we get
\begin{align*}
&\E[\sup_{s\leq u\leq t}|Y_u-Y_{s}|^p]\\
&\qquad\leq C\Big\{ |t-s|^p+\E\Big[
|t-s|^p \sup_{s\leq u\leq t} |Z_u|^{2p}+|t-s|^{\frac{p}{2}}\sup_{s\leq u\leq t}|Z_u|^{p}
\Big]\Big\}\\
&\qquad\leq C_p \Big\{ |t-s|^p+ |t-s|^{\frac{p}{2}}\Big\}.
\end{align*} The result follows.

\emph{Part ii)}: Theorem \ref{Z-sinfty-bounds} states that
$Z\in\cS^{2p}$. Therefore we are able to write, using Jensen's
inequality
\begin{align}\label{localaux:xpto}
\E\Big[\Big(\int_{t_i}^{t_{i+1}}|Z_t-Z_{t_i}|^2\udt\Big)^p\Big]\leq
\Delta^{p-1}\int_\ti^\tip\E[\,|Z_t-Z_{t_i}|^{2p}\,]\udt.\end{align}
In view of Theorem \ref{DY-continuity} and the subsequent
representation formula for $Z$ in terms the Malliavin derivatives of
$Y$ (see (\ref{z-rep})), we find an alternative way to express the
difference $Z_t-Z_{t_i}$ for $t\in[t_i, t_{i+1}]$ by writing
\begin{align}
Z_t-Z_\ti &= \nabla Y_t (\nabla X_t)^{-1}\sigma(X_t)-\nabla Y_\ti(\nabla X_\ti)^{-1}\sigma(X_\ti)
= I_1+I_2+I_3,\label{aux:loc4321}
\end{align} where $I_1=\Big(\nabla Y_t-\nabla Y_\ti\Big) (\nabla X_t)^{-1} \sigma(X_t)$,
$I_2=\nabla Y_\ti \Big( (\nabla X_t)^{-1}-(\nabla X_\ti)^{-1}\Big) \sigma(X_t)$ and $I_3=\nabla Y_\ti (\nabla X_\ti)^{-1}\Big(\sigma(X_t)-\sigma(X_\ti)\Big)$.

Estimates for $I_2$ and $I_3$ are easy to produce since they rely
mainly on $\|\nabla Y\|_{\cS^{2p}}<\infty$ and the results presented
in subsection \ref{subsec:sde-results}. We give details for $I_2$
and hints how to deal with $I_3$, remarking that its treatment is
very similar. H\"older's inequality combined with the growth
condition of $\sigma$ produce
\begin{align}
\E[\,|I_2|^{2p}]& \leq  C\,
\E\Big[\sup_{0\leq u\leq T}|\nabla_x Y_u|^{6p}\,\Big]^{\frac{1}{3}} \E\Big[\sup_{t_i\leq t\leq t_{i+1}}|(\nabla X_t)^{-1}-(\nabla X_\ti)^{-1}|^{6p}\,\Big]^{\frac{1}{3}}\,
\E\Big[\sup_{0\leq u\leq T}|X_u|^{6p}\,\Big]^{\frac{1}{3}}\nonumber \\
&\leq C\, \Delta^{{3p}\frac{1}{3}}=C\, \Delta^{p}\label{aux:loc123}.
\end{align}
For the last line we use (\ref{supnorm-X}), (\ref{delta-nablaX-1})
and $\|\nabla Y\|_{\cS^{2p}}<\infty$. For  $I_3$, the method is
similar: instead of (\ref{supnorm-X}) and (\ref{delta-nablaX-1}) we
have to use (\ref{delta-X}) and (\ref{supnorm-gradX}).

We next estimate $I_1$. Using Fubini's Theorem and H\"older's
inequality we get
\begin{align*}
&\int_\ti^\tip\IE\big[\,|I_1|^{2p}\,\big]\ud t = \IE\big[\,\int_\ti^\tip|I_1|^{2p}\ud t\,\big]\\
&\qquad\leq
\IE\Big[\int_\ti^\tip |(\nabla X_t)^{-1}|^{4p} \ud t\Big]^{\frac12}\,
\IE\Big[\int_\ti^\tip |\sigma(X_t)|^{4p} \ud t\Big]^{\frac12}\IE\Big[\sup_{\ti\leq t\leq \tip}|\nabla Y_t -\nabla Y_\ti |^{2p}\Big].
\end{align*}
We can simplify the integral terms by estimating the integrands by
their suprema over the intervals. Using the linear growth condition
on $\sigma$ combined with (\ref{supnorm-X}), (\ref{supnorm-gradX}),
we show in this way that the first two expectations on the right
hand side are bounded by $C \Delta^{1/2}$ each. Applying an
appropriate version of (\ref{localaux:xpto}), and using the previous
inequalities, we infer
\begin{align*}
\Delta^{p-1}\sum_{i=0}^{N-1}\int_\ti^\tip\IE\big[\,|I_1|^p\,\big]\ud t
&\leq
C \Delta^p \sum_{i=0}^{N-1} \IE\Big[\sup_{\ti\leq t\leq \tip}|\nabla Y_t -\nabla Y_\ti |^{2p}\Big].
\end{align*}
It remains to estimate $\nabla Y_t -\nabla Y_\ti$ for $t\in[t_i,
t_{i+1}]$ using the BSDE (\ref{diff-bsde}). For $p\geq 1$, the
inequalities of Doob and Burkholder-Davis-Gundy combine with {{\bf
HX1}} and {{\bf HY1}} in the same fashion as in part \emph{i)} to
yield for $\Theta=(X,Y,Z)$ and $\nabla \Theta = (\nabla X,\nabla
Y,\nabla Z)$
\begin{align*}
&\sum_{i=0}^{N-1}\IE\Big[\sup_{\ti\leq t\leq \tip}|\nabla Y_t -\nabla Y_\ti |^{2p}\Big]\\
&\quad\leq C\sum_{i=0}^{N-1}\IE\Big[
\Big(\int_\ti^\tip
|\langle (\nabla f)(s,\Theta_s), \nabla \Theta_s \rangle| \ud s\Big)^{2p} +\Big(\int_\ti^\tip |\nabla Z_s|^2\ud s\Big)^{p}\Big]\\
&\quad \leq C\,\IE\Big[ \Big(\int_0^T|\langle (\nabla
f)(s,\Theta_s), \nabla \Theta_s \rangle| \ud
s\Big)^{2p}+\Big(\int_0^T |\nabla Z_s|^2\uds\Big)^{p} \Big].
\end{align*}
For the last line we interchange summation and expectation and apply
(\ref{ineq:discrete-newton}). We now use the growth condition of
{\bf HY1} combined with the fact that $X,Y,Z,\nabla X, \nabla
Y\in\cS^{2p}$ and $\nabla Z\in \cH^{2p}$. Therefore
\begin{align*}
\sum_{i=0}^{N-1}\IE\Big[\sup_{\ti\leq t\leq \tip}|\nabla Y_t -\nabla
Y_\ti |^{2p}\Big]<\infty,
\end{align*}
which obviously implies
\begin{align*}
\Delta^{p-1}\sum_{i=0}^{N-1}\int_\ti^\tip\IE\big[\,|I_1|^p\,\big]\ud
t \leq C \Delta^p.
\end{align*}
Finally we inject (\ref{aux:loc123}) and the above inequality into
(\ref{localaux:xpto}) (according to (\ref{aux:loc4321})), to obtain
the second assertion of the Theorem:
\begin{align*}
\sum_{i=0}^{N-1} \E\Big[\Big(\int_{t_i}^{t_{i+1}}|Z_t-Z_{t_i}|^2\udt\Big)^p \Big] & \leq C \Delta^{p-1}\sum_{i=0}^{N-1} \int_{t_i}^{t_{i+1}}\E\big[\,|I_1|^{2p}+|I_2|^{2p}+|I_3|^{2p}\big]\udt\\
& \leq C \Delta^{p-1}\big(\Delta+ 2  \Delta^p \big)\leq C \Delta^p.
\end{align*}

\emph{Part iii)}: Theorem \ref{DY-continuity} states the map $t\mapsto
D_tY_t$ is a continuous version of $Z$. Hence we are able to express for
$s,t\in[0,T]$ the difference $Z_t-Z_s$ by Malliavin derivatives of
$Y$, and its moments for $p\ge 2$ by
\[ E[\,|Z_t-Z_s|^p]\leq C (E[\,|D_tY_t-D_sY_t|^p]+ E[\,|D_sY_t-D_sY_s|^p]),\quad \textrm{with }s\leq t. \]
We estimate both expressions on the right hand side separately. The arguments we use are similar to the ones in \emph{Part ii)}.

From Lemma \ref{kolm-for-DY-DZ} we have $\IE[\,|D_tY_t-D_sY_t|^p\,]\leq C|t-s|^{\frac{p}{2}}$. For the other term, a simple calculation using BSDE
(\ref{aux:1order-mall-dif-bsde}) yields
\begin{align*}
D_s Y_t-D_s Y_s
= \int_s^t \big[\langle \nabla f(\Theta_u), D_s \Theta_u \rangle\big]\ud u-\int_s^tD_s Z_u\ud W_u.
\end{align*}
By Doob's and Burkholder-Davis-Gundy's inequalities we have for
$p\geq 2$
\begin{align*}
\E[\sup_{s\leq u\leq t}|D_s Y_u-D_s Y_s|^p]&\leq C\,\IE\Big[\Big(\int_s^t \big[ (1+|Y|+|Z|^2)|D_sX_u|+|D_sY_u|\\
&\hspace{2cm}+ (1+|Z_u|)\,|D_sZ_u|\,\big]\ud u\Big)^p+\Big(\int_s^t |D_s Z_u|^2\ud u\Big)^{\frac{p}{2}}\Big]\\
&\leq C\,\big\{ |t-s|^p+|t-s|^{\frac{p}{2}}\big\}.
\end{align*}
This last line follows, because all the integrand processes belong
to $\cS^{p}$ for all $p\geq 2$ (see Theorem \ref{Z-sinfty-bounds}).
Combining the two above estimates we have $E[\,|Z_t-Z_s|^p]\leq
C\,|t-s|^{\frac{p}{2}}$ as intended. Kolmogorov's continuity
criterion yields the continuity statement.
\end{proof}

\subsection{The path regularity Theorem for qgBSDE}

Now let $\pi$ be a partition of the interval $[0,T]$ with $N$ points and mesh size $|\pi|$. We define a set of random variables
\begin{align}\label{Z-bar-ti-pi}
\bar{Z}^\pi_\ti&=\frac1h\E\Big[\int_\ti^\tip
Z_s\uds\big|\cF_\ti\Big], \textrm{ for all partition
points } t_i,\ 0\le i\le N-1,
\end{align}
where $Z$ is the control process in the solution of qgFBSDE (\ref{fbsde-sde}), (\ref{fbsde-bsde}) under {\bf HX0} and {\bf HY0}. It is not difficult to show that $\bar{Z}^\pi_\ti$ is the best $\cF_\ti$-adapted $\cH^2([\ti,\tip])$ approximation of $Z$, i.e.
\[
\E\Big[ \int_\ti^\tip |Z_s-\bar{Z}^\pi_\ti|^2 \uds\Big]= \inf_{Z_i \in L^2(\Omega,\cF_\ti)} \IE\Big[\int_\ti^\tip |Z_s-Z_i|^2 \uds\Big].
\]
Let now $\bar{Z}^\pi_t = \bar{Z}^\pi_\ti$ for $t\in[\ti, \tip[, 0\le i\le N-1.$ It is equally easy to see that $\bar Z^\pi$ converges to $Z$ in $\cH^2$ as $|\pi|$ vanishes: since $Z$ is adapted there exists an adapted family of processes $Z^\pi$ indexed by our partition such that $Z^\pi_t=Z_\ti$ for $t\in [\ti,\tip)$ and that $Z^\pi$ converges to
$Z$ in $\cH^2$ as $|\pi|$ goes to zero. Since $\{\bar Z^\pi\}$ is the best $\cH^2$-approximation of $Z$, we obtain
\[
\| Z-\bar Z^\pi  \|_{\cH^2}\leq \|Z-Z^\pi \|_{\cH^2}\to 0,\ \textrm{ as }\ h\to 0.
\]
As an immediate corollary of \emph{ii)} in the previous Theorem we get the extension to the setting of drivers with quadratic growth of the famous {Theorem 3.4.3} in \cite{phd-zhang}.
Let $p=1$ in Theorem \ref{theo:path-reg}. Then
\begin{theo}\label{zhangs-path-reg-theo-to-qgfbsde}Assume {\bf HX1} and {\bf HY0}. Assume further that condition (\ref{unif-ellipt}) holds and that $g$ satisfies a standard Lipschitz condition with Lipschitz constant $M$. Then there exists a constant $C$
such that for any partition $\pi = \{t_0<\cdots<t_N\}$ of the interval $[0,T]$ with mesh size $|\pi|$ we have

%Under {\bf HX1} and {\bf HY1} and for each $0\leq i\leq N-1$ we have
\[
\max_{0\leq i\leq N-1}\Big\{\sup_{ t\in [\ti,\tip)} \E\Big[\,|Y_t-Y_\ti|^2\Big]\Big\}+
\sum_{i=0}^{N-1} \E\Big[\int_\ti^\tip |Z_s-\bar{Z}^\pi_\ti|^2\uds \Big]\leq C |\pi|.
\]
\end{theo}
\begin{proof}
Let $(f^\varepsilon)_{\varepsilon>0}$ and $(g^\varepsilon)_{\varepsilon>0}$ be two families of $C^\infty_0$ functions obtained by canonically regularizing $f$ and $g$ respectively,
and such that
\[
\lim_{\varepsilon\to0} \Big\{
\sup_{(t,x,y,z)\in[0,T]\times \IR^m\times \IR \times \IR^d} |f^\varepsilon(t,x,y,z)-f(t,x,y,z)|
+\sup_{x\in\IR^m} |g^\varepsilon(x)-g(x)|
\Big\}=0.
\]
Since $f$ satisfies condition {\bf HY0} and $g$ is uniformly Lipschitz, for any $\varepsilon>0$ both $g^\varepsilon$ and $f^\varepsilon$ satisfy condition {\bf HY1}. Let us denote $(X,Y^\varepsilon,Z^\varepsilon)$ the solution of (\ref{fbsde-sde}) and \[
Y^\varepsilon_t = g^\varepsilon(X_T)
+\int_t^T f^\varepsilon(s, X_s,Y^\varepsilon_s, Z^\varepsilon_s)\uds -\int _t^T Z^\varepsilon_s \udws,\quad t\in[0,T],\ \varepsilon>0.
\]
Then all the results of section \ref{sec:reg-in-time} proved so far hold for the pair $(Y^\varepsilon, Z^\epsilon)_{\varepsilon>0}$. Using Lemma \ref{comparison.theo} we can conclude that for any $p\geq 1$
\[
\lim_{\varepsilon\to 0} \Big\{ \|Y^\varepsilon-Y\|_{\cS^{2p}}+ \|Z^\varepsilon-Z\|_{\cH^{2p}}
\Big\}=0.
\]
We next apply Theorem \ref{theo:path-reg}. After a careful inspection of the arguments of its proof, we find a positive constant $C$ independent of $\varepsilon>0$ such that
for any partition $\pi = \{t_0<\cdots<t_N\}$ of the interval $[0,T]$ with mesh size $|\pi|$ we have
\[
\max_{0\leq i\leq N-1}\Big\{\sup_{ t\in [\ti,\tip)} \E\Big[\,|Y^\varepsilon_t-Y^\varepsilon_\ti|^2\Big]\Big\}+
\sum_{i=0}^{N-1} \E\Big[\int_\ti^\tip |Z^\varepsilon_s-\bar{Z}^{\varepsilon,\pi}_\ti|^2\uds \Big]\leq C |\pi|,
\] with the set $\{ \bar{Z}^{\varepsilon,\pi}_\ti \}_{\ti\in\pi}$ given as in (\ref{Z-bar-ti-pi}) for the process $(Z^\varepsilon_t)_{t\in[0,T]}.$
We finally apply Fatou's lemma to obtain
\[
\sum_{i=0}^{N-1} \E\Big[\int_\ti^\tip |Z_s-\bar{Z}^{\pi}_\ti|^2\uds \Big]
\leq
\liminf_{\varepsilon\to 0}
\sum_{i=0}^{N-1} \E\Big[\int_\ti^\tip |Z^\varepsilon_s-\bar{Z}^{\varepsilon,\pi}_\ti|^2\uds \Big]\leq C |\pi|.
\]
A similar argument holds for the difference $(Y^\varepsilon_t-Y^\varepsilon_\ti)$, $\ti\in\pi$ and $t\in[\ti,\tip)$.

The result now follows.
\end{proof}

\section{Numerics for qgFBSDE - a truncation procedure}

A common method to deal with non-linearities or unbounded functions consists in truncating them. In our BSDE (\ref{fbsde-bsde}), the driver has a
quadratic nonlinearity in $z$. Our truncation of the nonlinear driver will relate BSDE with drivers of quadratic growth BSDE with globally Lipschitz
drivers. For this type of BSDE numerical schemes are readily available (see \cite{04bouchardtouzi}, \cite{phd-elie} and references therein), and the
error committed in the numerical approximation for BSDE with Lipschitz driver is well known. So to fully analyze the error related to successive
approximations in the case of BSDE with drivers of quadratic growth, it only remains to provide an estimate for the error arising from the truncation.
This is what we propose to do in this section.

Emphasizing once more the point made in the introduction, we remark that the convergence
rate for numerical schemes of truncated qgFBSDE is well known and
produces an exponential dependence on the truncation level. Despite this fact,
the result presented in Theorem \ref{theo:trunc-conv-rate} implies that one can obtain a high convergence order for
the truncation procedure. This modestly mitigates the exponential dependence of the convergence order of the scheme on the truncation level.

For our qgFBSDE system (\ref{fbsde-sde}), (\ref{fbsde-bsde}) we assume that {\bf HX1} and {\bf HY1} hold. In this section the diffusion process $X$
appearing in the BSDE's terminal condition and driver plays a secondary role, especially in the calculations we will be presenting.

To truncate the driver of quadratic growth, we use the already familiar sequence $\{h_n\}_{n\in\IN}$ defined in (\ref{trunc-family}). To justify
that this sequence indeed does the job, we will need (\ref{Z-moment-time}) to make our calculations work. At first we have to justify that the truncated FBSDE obtained in this way satisfies {\bf HX1} and {\bf HY1}.

Recalling the driver of BSDE (\ref{fbsde-bsde}), we define a family of functions $f_n:[0,T]\times \R^m \times \R\times\R^d\to \R$, $(\omega,
t, x, y, z) \mapsto f(\omega,t,x,y,h_n(z))$ and with it, the family of truncated FBSDE
\begin{equation}
\label{tr:quad-bsde-truncated} Y^n_t = \xi +\int_t^T
f(s,X_s,Y^n_s,h_n(Z^n_s))   \uds -\int_t^T Z^n_s\udws,
\end{equation}
with $\xi = g(X_T)$. The solution process of (\ref{fbsde-bsde}) is denoted by $(Y,Z)$ and the solution process of its truncated counterpart (\ref{tr:quad-bsde-truncated}) by $(Y^n,Z^n)$.
Furthermore, we recall that in (\ref{trunc-family})
$(h_n)_{n\in\IN}$ was defined as a sequence of $C^1$ functions, and that by Theorem \ref{theo:moment-estimates-special-class} we have
\begin{align}
\nonumber
&\max\Big\{
\sup_{n\in\IN}\| Z^n*W \|_{BMO} ,\ \| Z*W \|_{BMO}\Big\}\\
\label{unif-bmo-norm-trunc}
& \hspace{4cm}
\leq
\frac{4+6M^2T}{3M^2} \exp \Big\{6M\|\xi\|_{L^\infty}+MT\Big\}<\infty.
\end{align}
This means that the martingales of the sequence $(Z^n*W)_{n\in\IN}$ satisfy the reverse H\"older inequality with an exponent $\barr$ independent of $n$ (see subsection \ref{bmo-subsection}, and also Remark \ref{barr.indep.n}).
\begin{remark}\label{trunc-numerics-remark}
If (\ref{fbsde-bsde}) satisfies {\bf HX1} and {\bf HY1}, by inspection of the hypotheses it is easy to see that family (\ref{tr:quad-bsde-truncated}) also satisfies {\bf HX1} and {\bf HY1} uniformly in $n$. %Furthermore, both {\bf HX1} and {\bf HY1} are satisfied independently of $n$.
This means that the results on differentiability in subsection \ref{subsec:qgBSDE} and on continuity and bounds in section \ref{subsec:cont-bounds} are available for the truncated BSDE (\ref{tr:quad-bsde-truncated}) as well.
\end{remark}
The proof of our result on the truncation error relies on Markov's inequality and (\ref{Z-moment-time}). The convergence rate will depend on a parameter which arises from the inverse
H\"older inequality and is related to (\ref{unif-bmo-norm-trunc}).
\begin{theo}\label{theo:trunc-conv-rate}
Assume that {{\bf HX1}} and {{\bf HY1}} are satisfied, and let $(Y,Z)$ and $(Y^n,Z^n)$ be solutions of (\ref{fbsde-bsde}) and (\ref{tr:quad-bsde-truncated}) respectively.
Then for any $p\geq 1$ and $\beta \geq 1$ there exist positive finite constants $C_p$ and $D_{\beta}$ such that for $n\in\IN$
\[
\E\Big[\sup_{t\in[0,T]} |Y^n_t-Y_t|^{2p}\Big]+\E\Big[\Big(\int_0^T
|Z^n_s-Z_s|^2\ud s\Big)^p\Big]\leq C_p\, D_{\beta}\, {n^{-\frac{\beta}{2 \barq}}}.
\]
The constant $\barq$ is the H\"older conjugate of
$\barr\in(1,\infty)$ which is related to the estimate (\ref{unif-bmo-norm-trunc}) according to Remark \ref{def:r-bar}. The constant $D_{\beta}$ is given by
$D_\beta=\big(\sup_{n\in\IN}\|Z^n\|_{\cS^{2\beta \barq}(\IP)} \big)^{\frac{\beta}{2\barq}}.$ The constant $C_p$ is independent of $\beta\geq 1$.
%Then for all $p\geq 1$ there exists a positive constant $C_p$ such that for all $n\in\IN$
%\[
%\E\Big[\sup_{t\in[0,T]} |Y^n_t-Y_t|^{2p}\Big]+\E\Big[\Big(\int_0^T
%|Z^n_s-Z_s|^2\ud s\Big)^p\Big]\leq C_{p,\barq}\, \frac1{n^{34}}.
%\]
%The constant $\barq$ is the H\"older conjugate of
%$\barr\in(1,\infty)$ which is related to the BMO martingale $Z * W$ as explained in Remark \ref{def:r-bar}.
\end{theo}
\begin{proof} Let $n\in\IN$ and $t\in[0,T]$. As usual we have to rely on an a priori estimate for the difference of original and truncated BSDE. To this end we use the notation,
methods and arguments of the proofs of Lemmas \ref{apriori.estimate}
and \ref{comparison.theo} (see Section \ref{sec:bmo-BSDE}), without
repeating all the details. For
\[
b^n_t=\frac{f_n(t,X_t,Y^n_t,Z^n_t)-f(t,X_t, Y^n_t,
Z_t)}{|Z^n_t-Z_t|^2}(Z^n_t-Z_t)\1_{\{Z^n_t\neq Z_t\}},\quad
t\in[0,T], n\in\IN,\]
the associated equivalent measure obtained
after measure change by subtracting the drift related to $b^n$ will
be denoted by $\IQ^{b_n}$. The superscript will be omitted for convenience. Using H\"older's inequalty and for some positive real constant $C$ we obtain, %similarly to the last step
%For a positive constant $C$ and Following the last step of the proof of Lemma \ref{apriori.estimate} or \ref{comparison.theo} we obtain through the use of H\"older's inequality
\begin{align}
\nonumber
&\E^\IP\Big[\sup_{t\in[0,T]} |Y_t-Y^n_t|^{2p}+\Big(\int_0^T
|Z_s-Z^n_s|^2\ud s\Big)^{p}\Big]\\
\label{aux:first-ineq}
&\qquad\qquad \leq D~ \E^\IQ\Big[\sup_{t\in[0,T]} |Y_t-Y^n_t|^{2p\barq}+\Big(\int_0^T
|Z_s-Z^n_s|^2\ud s\Big)^{p\barq}\Big]^{\frac{1}{\barq}},
\end{align}
with $D= \sup_{n\in\IN} \E^\IP[ \cE(b_n*W)^\barr ]^{\frac{1}{\barr}}<\infty$.
That $D$ is finite follows from (\ref{unif-bmo-norm-trunc}) combined with part 2) of Lemma \ref{bmoeigen}.

We continue with the estimation of (\ref{aux:first-ineq}). Following the proof of Lemma \ref{apriori.estimate} or
Lemma \ref{comparison.theo} (see (\ref{temp-Q-bounds})), there exists a positive constant $C$ such that
\begin{align}
\nonumber
&\E^\IQ\Big[\sup_{0\leq t\leq T}|Y_t-Y_t^n|^{2p\barq}+\Big( \int_0^T|Z_s-Z_s^n|^{2}\uds \Big)^{p\barq}\Big]^{\frac{1}{\barq}}\\
\nonumber
&\qquad \leq C\,
\E^\IQ\Big[ \Big( \int_0^T \big|f\big(s,Y^n_s,Z^n_s\big)-f\big(s,Y^n_s,h_n(Z^n_s)\big)\big| \uds \Big)^{2p\barq}\Big]^{\frac{1}{\barq}}\\
\nonumber
&\qquad \leq C\,
\E^\IQ\Big[ \Big( \int_0^T M\big(1+|Z_s^n|+|h_n(Z^n_s)|\big) \big|Z^n_s-h_n(Z^n_s)\big| \uds \Big)^{2p\barq}\Big]^{\frac{1}{\barq}}\\
\nonumber
&\qquad \leq C\,
\E^\IQ\Big[\Big( \int_0^T \big|M(1+|Z_s^n|+|h_n(Z^n_s)|)\big|^2 \uds\Big)^{2p\barq}\Big]^{\frac{1}{2\barq}} \E^\IQ\Big[ \Big(\int_0^T |Z^n_s-h_n(Z^n_s)|^2
\uds \Big)^{2p\barq}\Big]^{\frac{1}{2\barq}}\\
\label{aux-pre-markov-ineq}
&\qquad \leq C\,
\E^\IQ\Big[ \Big( \int_0^T |Z^n_s-h_n(Z^n_s)|^2 \uds \Big)^{2p\barq} \Big]^{\frac{1}{2\barq}},
\end{align}
where we made use of the growth assumption on $f$ stated in {\bf HY1}, H\"older's inequality and (\ref{temp-Q-bounds}).

A closer look at the properties of $h_n$ reveals that for any $n\in\IN$ and $s\in[0,T]$ we have
\[|Z_s^n-h_n(Z^n_s)|^2\leq 4 |Z^n_s|^2\1_{\{|Z^n_s|> n\}}.\]
In view of this inequality, an explicit convergence rate can be obtained if the term $\1_{\{|Z^n_s|> n\}}$ is explored. Because (\ref{Z-moment-time}) holds we can use Markov's inequality for this purpose.

As pointed out in Remark \ref{trunc-numerics-remark}, the validity of {\bf HX1} and {\bf HY1} for the family of drivers used entitles us to employ the crucial (\ref{Z-moment-time}) of Theorem \ref{Z-sinfty-bounds}.
We now develop (\ref{aux-pre-markov-ineq}) using sequentially H\"older's inequality, Jensen's inequality and  Fubini's theorem, to obtain
\begin{align*}
\nonumber
&C\,\E^\IQ\Big[\Big(\int_0^T |Z^n_s-h_n(Z^n_s)|^2\ud s\Big)^{2p\barq}\Big]^{\frac{1}{2\barq}}\\
\nonumber
&\quad \leq C\, \E^\IQ\Big[\Big( \int_0^T |Z^n_s|^4\uds \Big)^{2p\barq}\Big]^\frac{1}{4\barq}
\E^\IQ\Big[\Big(\int_0^T \1_{\{|Z^n_s|> n\}} \uds \Big)^{2p\barq}\Big]^\frac{1}{4\barq}
%\leq C\, \E^\IQ\Big[\Big( \int_0^T \1_{\{|Z^n_s|> n\}} \uds \Big)^{4p}\Big]^\frac{1}{2}\\
\leq C\, \E^\IQ\Big[\int_0^T \1_{\{|Z^n_s|>
n\}} \ud s\Big]^\frac{1}{4\barq}\\
%\label{aux:at-markov}
&\quad \leq C\, \Big(\int_0^T \E^\IQ[\1_{\{|Z^n_s|> n\}}]\uds\Big)^\frac{1}{4\barq}
= C\,\Big(\int_0^T \IQ\big[\{|Z^n_s|> n\}\big]\uds \Big)^\frac{1}{4\barq}.
\end{align*}
Applying Markov's inequality we obtain for some $\beta\geq 1$
\begin{align*}
C \Big( \int_0^T\frac{1}{n^{2\beta }}\E^\IQ[\,|Z_s^n|^{ 2\beta }]\uds \Big)^\frac{1}{4\barq}
= C\, n^{-\frac{\beta}{2\barq} }\, \E^\IQ\Big[\int_0^T |Z_s^n|^{2\beta}
\uds\Big]^\frac{1}{4\barq}
\leq C\,D\, \IE^\IP\Big[\sup_{t\in[0,T]} |Z_t|^{2\beta \barq} \Big]^{\frac{1}{4\barq^2}}\,n^{-\frac{\beta}{2\barq}},
\end{align*}
with $D$ as in inequality (\ref{aux:first-ineq}). We emphasize that the constant $C$ which varies from line to line depends on $p$, $T$, $\barr$ and $\barq$, but \emph{not} on $n$ or $\beta$.

By construction, it is clear that
\begin{align*}
\E^\IP\Big[\sup_{t\in[0,T]} |Y_t-Y^n_t|^{2p}+\Big(\int_0^T
|Z_s-Z^n_s|^2\ud s\Big)^{p}\Big]
\leq
C\, \Big(\sup_{n\in\IN}\|Z^n\|_{\cS^{2\beta \barq}(\IP)} \Big)^{\frac{\beta}{2\barq}} \, n^{-\frac{\beta}{2\barq}},
\end{align*}
with a positive constant $C$ independent of $\beta$ and $n$.

We finish this proof with an argument establishing the finiteness of $\sup_{n\in\IN}\|Z^n\|_{\cS^{\gamma}}$ for $\gamma>2$.
Having in mind Remark \ref{trunc-numerics-remark} we can apply, for every $n\in\IN$, Theorem \ref{theo:1st-order-diff-qgbsde} to BSDE (\ref{tr:quad-bsde-truncated}).
We obtain that for each $n$ that the pair $(Y^n,Z^n)$ is differentiable with derivatives given by $(\nabla Y^n, \nabla Z^n).$
The derivatives satisfy BSDE (\ref{diff-bsde}) with driver $f$ replaced by the corresponding driver $f_n$ (see BSDE (\ref{tr:quad-bsde-truncated})).

Given the properties of the sequence $(h_n)_{n\in\IN}$ and inequality  (\ref{unif-bmo-norm-trunc}), we can apply Lemma \ref{apriori.estimate} to the BSDE for $(\nabla Y^n, \nabla Z^n)$
and easily obtain that for any $\gamma \geq 2$,
$\sup_{n\in\IN} \|\nabla Y^n \|_{\cS^\gamma}<\infty$.

With arguments similar to those used to prove  (\ref{Z-moment-time}), it follows that for any $\gamma\geq 2$ we have $\sup_{n\in\IN}\|Z^n\|_{\cS^{\gamma}}<\infty$.
\end{proof}
\appendix
\section{Appendix}
In this appendix we give the technical details left out in section
\ref{section:3} in the proof of second order Malliavin
differentiability of the solution processes of a BSDE the driver of
which satisfies Lipschitz conditions.

The techniques we will use are not new. They are based on a Picard
iteration argument. It does not only give existence and uniqueness
of solutions. It also allows to establish Malliavin
differentiability in each step for the respective approximation of
the solution. By means of a contraction argument in a suitable
Sobolev norm, Malliavin smoothness is carried over to the solution
in the limit. In contrast to previous applications, here the scheme
deals with an equation that already has a Malliavin derivative as
its solution.

We start with canonical coefficients that are given by an
$\cF_T$-measurable random variable $\xi$ and a measurable function
$f:\Omega\times[0,T]\times[0,T]\times\IR\times \IR^d\to \IR$, such
that
\[f(\cdot, s, u, y, z)=a_{u,s}(\cdot)+b_s(\cdot)\, y + \langle c_s(\cdot), z\rangle.\]
For the remainder we omit the dependence of the coefficients on $\omega\in\Omega$. The coefficient functions defining this driver will be supposed to
satisfy the following assumptions.
\bit
\item[(A1)]%\label{hypothesis-B1}
$b:\Omega\times[0,T]\to \IR$ and $c:\Omega\times[0,T]\to \IR^d$ are measurable $(\cF_t)$-adapted processes bounded by a constant
$M>0$.

$a:\Omega\times[0,T]\times[0,T]\to \IR$ satisfies $\sup_{0\leq u\leq
T\,}\lVert a_{u,\cdot} \lVert_{\cS^{2p}}<\infty$ for all $p\geq 1$.
For each fixed $u\in[0,T]$ the process $a_{u,t}$ is progressively measurable.

\item[(A2)]%\label{hypothesis-B2}
$\xi$ is a Malliavin differentiable, $\cF_T$-measurable bounded random variable with Malliavin derivative
given by $D \xi$ satisfying
\[\sup_{0\leq u\leq T} \lVert D_u \xi \lVert_{L^{2p}} <\infty,\ \textrm{for all } p\geq
1.\]
\item[(A3)]%\label{hypothesis-B3}
$a_{u,t}, b_t, c_t$ are Malliavin differentiable for all $u\in[0,T]$ and $t\in[0,T]$. Measurable versions of their Malliavin derivatives
are respectively given by $D_v a_{u,t}$, $D_v b_{t}$ and $D_v c_{t}$
for $v \in[0,T]$ such that for all $p\geq 1$
\[
\sup_{0\leq v,u\leq T} \E\Big[\Big(\int_0^T \big[\,|D_v a_{u,s}|^2+
|D_v b_s|^2 + |D_v c_s|^2\big] \uds\Big)^p\Big]<\infty.
\]
\item[(A4)]%\label{hypothesis-B4}
For all $u\in[0,T]$, $D_u \xi$ is Malliavin differentiable, with second order derivative given by $D_vD_u\xi, v,u\in[0,T],$ and
satisfying $\sup_{0\leq u,v\leq T}\|D_v D_u \xi\|_{L^{p}}<\infty$
for all $p\geq 1$.
\eit

Under these assumptions we consider the following backward
stochastic differential equation
\begin{align}\label{appendix:1order-mall-dif-bsde}
U_{u,t} &= 0, \qquad V_{u,t} = 0,\qquad t\in [0, u),\nonumber\\
U_{u,t} &= D_u \xi
 - \int_t^T  V_{u,s} \ud W_s + \int_t^T f(s, u, U_{u,s}, V_{u,s} )\ud s, \qquad t\in
[u,T].
\end{align}

\begin{theo}\label{appendixB:2-order-mall-diff-lip-BSDE}
Under (A1) and (A2), the BSDE (\ref{appendix:1order-mall-dif-bsde})
has a unique solution $(U,V)$ in $\cS^{2p} \times \cH^{2p}$ for $p\geq 1$. Furthermore, if (A3) and (A4) hold, then $(U,V)$ is Malliavin differentiable and a version of
$\{(D_v U_{u,t},D_v V_{u,t});0\leq v\leq u\leq t\leq T\}$ satisfies
for any $0\leq v\leq u\leq t\leq T$
\begin{align}\label{appendix:2order-mall-dif-bsde}
D_vU_{u,t} &= D_vD_u \xi  -\int_t^T D_vV_{u,s}\ud W_s \nonumber\\
&\qquad +\int_t^T \Big[(D_v f)(s,u, U_{u,s}, U_{u,s})+ \big\langle
\nabla f(s,u,U_{u,s}, U_{u,s}),
(D_vU_{u,s},D_vU_{u,s})\big\rangle\Big]\ud s.
\end{align}
A version of $\{V_{u,t};0\leq u\leq t\leq T\}$ is given by $\{D_t U_{u,t};0\leq u\leq t\leq T\}$.
\end{theo}
\begin{proof}
For the sake of notational simplicity and clarity, we provide a
proof for the case $d=1$. This proof splits into two steps. In the
first one, we are concerned with existence and uniqueness of
solutions for (\ref{appendix:1order-mall-dif-bsde}) using a Picard
iteration. In the second step we prove Malliavin differentiability.
To this end, we show that the sequence arising in the Picard scheme
is in fact Malliavin differentiable, and by contraction that its
limit must be the Malliavin derivative of the solution process
constructed in the first part.

\emph{Part i):} To simplify notation we refer to $D_u \xi$ as
$\xi_u$ if there is no ambiguity.

We have a ``standard'' BSDE with Lipschitz continuous driver and a
smooth terminal condition.  The usual arguments for existence and
uniqueness in this setting are well known. We recall the Picard
iteration argument of the proof of {Proposition 5.3} in
\cite{97KPQ}. Let $(U_{u,t}^0, V_{u,t}^0)=(0,0)$ and for $k\ge 0$
define recursively the pair $(U_{u,t}^{k+1}, V_{u,t}^{k+1})$ as the
solution of
\begin{align*}
U_{u,t}^{k+1}&= D_u \xi  -\int_t^T V_{u,s}^{k+1}\ud W_s+\int_t^T f(s,u,U_{u,s}^{k},V_{u,s}^{k}) \ud s,\quad\textrm{for } 0\leq u\leq t\leq T.
\end{align*}
Under (A1) and (A2) the iteration scheme is well defined and the
following moment estimates hold for all $k\in\IN$ and $p\geq 2$ (see
{Proposition 2.1} in \cite{97KPQ}):
\begin{align*}%\label{mom-estim-mal-dif-seq-k}
\sup_{0\leq u\leq T}\Big\{  \| U_u^k \|_{\cS^p}^p+\| V_u^k\|_{\cH^p}^p\Big\} \leq C \sup_{0\leq u\leq T}\Big\{ \| D_u \xi\|_{L^p}^p
+\| a_{u,\cdot} \|_{\cH^p}^p\Big\}<\infty.
\end{align*}
Along the classical lines of the argument\footnote{Applying the It\^o
formula to $e^{\beta t}(U^{k+1}_{u,t}-U^k_{u,t})^2, t\in[0,T],$ one
proves in the usual fashion norm contraction of the sequence through
the a priori estimates. As this argument is well known we omit it.}
of {Corollary 2.1} in \cite{97KPQ} we obtain: $(U^k,V^k)\to (U,V)$
$\ud \IP \otimes \ud t \otimes \ud u$ a.e. as well as
\begin{align*}
\lim_{k\to \infty}\ \sup_{0\leq u\leq T}\Big\{  \| U^{k+1}_u- U_u \|_{\cS^{2p}}+\| V^{k+1}_u-V_u\|_{\cH^{2p}}\Big\} = 0,\quad \textrm{for all }p\geq 1.
\end{align*}

\emph{Part ii):} In what follows we prove that the sequence $(U^k,
V^k)_{k\in\IN}$ is Malliavin differentiable and its Malliavin
derivatives converge to a version of  the Malliavin derivative of
$(U,V)$, the arguments we use are close to the ones in \cite{{phd-elie}}. The proof is done recursively, starting with the initial step.
From (A3) and (A4), the differentiability of $\xi_u$ and $a_{u,t}$
implies that for all $0\leq u\leq t\leq T$ the process $\E\big[
\xi_u +\int_t^T f(s,u,0,0)\ud s|\cF_t\big]\in \ID^{1,2}$ and hence
\begin{align*}
\E\big[ \xi_u +\int_t^T f(s,u,0,0)\ud s|\cF_t\big]&=U^1_{u,t} \in \ID^{1,2}.
\end{align*}
Since $\xi+\int_t^T f(s,u,0,0)\ud s-U^1_{u,t}=\int_t^T
V^1_{u,s}\udws$, Lemma 5.1 of \cite{97KPQ} implies
$V^1_{u,t}\in\ID^{1,2}$. For the recursive step, we next show that
if $(U^k_{u,t},V^{k}_{u,t})\in\ID^{1,2}$, then also
$(U^{k+1}_{u,t},V^{k+1}_{u,t})\in\ID^{1,2}$. Assume that
$(U^k_{u,t},V^{k}_{u,t})\in\ID^{1,2}$. Since $b,c\in\ID^{1,2}$, by
the rules of Malliavin calculus we have $\E\big[ \xi_u+\int_t^T
f(s,u, U^k_{u,s},V^k_{u,s})\ud s |\cF_t\big]\in\ID^{1,2}$ and hence
$U^{k+1}_{u,t}\in\ID^{1,2}$. Consequently for $\int_t^T
V^{k+1}_{u,s}\udws=\xi_u+\int_t^T f(s,u, U^k_{u,s},V^k_{u,s})\ud s
-U^{k+1}_{u,t}$ again Lemma 5.1 in \cite{97KPQ} yields
$V^{k+1}_{u,t}\in\ID^{1,2}$. Given these properties we have for
$0\leq v\leq u\leq t\leq T$
\begin{align*}
D_v U^{k+1}_{u,t} &= D_v \xi_u -\int_t^T D_v V^{k+1}_{u,s}\udws\\
&\quad+ \int_t^T\Big[ (D_v f)(s,u,U^k_{u,s},V^k_{u,s})+\big\langle
(\nabla f)(s,u,U^k_{u,s}, V^k_{u,s}), (D_v U^{k}_{u,s}, D_v
V^{k}_{u,s}) \big\rangle\Big]\ud s.
\end{align*}
We continue by showing that the sequence $(D_v U^{k}_{u,t},D_v
V^{k}_{u,t})$ converges and identify its limit as $(D_v U_{u,t},D_v
V_{u,t})$ which in addition is a solution of
(\ref{appendix:2order-mall-dif-bsde}).

If we assume that equation (\ref{appendix:2order-mall-dif-bsde}) has
a solution $(D_v U_u,D_v V_u)$ then the usual moment estimation
techniques combined with the current assumptions produce
\begin{equation}\label{aux:ddyddy-moment}
\sup_{0\leq v,u\leq T} \Big\{\|D_v U_u\|_{\cS^{2p}}+\|D_v
V_u\|_{\cH^{2p}}\Big\}< \infty, \qquad p\geq 1.\end{equation} Fix
$N\in\IN$ to be chosen later, fix $0\le v\le u\le T$, set
$\delta=T/N$ and define a partition $\tau_i=i\delta$ for
$i\in\{1,\ldots,N\}$. Then a priori estimates yield for $0\le i\le
N-1$
\begin{align}\nonumber
A^{k+1}_{u,v,i}&=\|D_v U^{k+1}_u-D_v U_u\|_{\cS^2([\tau_i, \tau_{i+1}])}^2+\| D_v V_u^{k+1}-D_v V_u \|_{\cH^2([\tau_i, \tau_{i+1}])}^2\\
\label{aux:iteration-ineq} &\qquad\leq C \left\{ \E\Big[ |D_v
U^{k+1}_{u,\tau_{i+1}}-D_v  U_{u,\tau_{i+1}}|^2\Big]+
B^{k}_{u,v,i}+C^{k}_{u,v,i} \right\}
\end{align}
with
\begin{align*}
B^k_{u,v,i}&= \big\|\,|D_v b|\,|U^{k}_{u}-U_u|+|D_v c|\,|V^{k}_{u}-V_u|\,\big\|_{\cH^2([\tau_i,\tau_{i+1}])}^2,\\
C^k_{u,v,i}&=
\E\Big[\Big(\int_{\tau_i}^{\tau_{i+1}}\Big[\,|b_s|\,|D_v
U^{k}_{u,s}-D_v U_{u,s}| +|c_s|\,|D_vV^{k}_{u,s}-D_v
V_{u,s}|\,\Big]\ud s\Big)^2\Big].
\end{align*}
Since both $b$ and $c$ are bounded, Jensen's inequality yields
\[
C^k_{u,v,i}\leq C \delta A^k_{u,v,i}
\]
and hence, an induction argument combined with
(\ref{aux:ddyddy-moment}), (\ref{aux:iteration-ineq}) and the
assumptions provides
\begin{align}\label{aux-A}\sup_{0\leq u,v\leq T}A^k_{u,v,i}<\infty,\quad \textrm{for all }k\geq 0.
\end{align}
To estimate $B^k_{u,v,i}$, note that according to (A3), $\sup_{0\leq
v\leq T}\big\{\|D_v b\|_{\cH^2}+\|D_v c\|_{\cH^2}\big\}<\infty$ and that
according to the first part of the proof $(U^{k}-U,V^{k}-V)\to 0$ in
$\cS^{2p}\times \cH^{2p}$, $p\geq 1$. Now choose $N$ large enough to guarantee
$\alpha = C \delta < 1$. Therefore for any $\eta>0$ one finds a
$K^*\geq 0$, independent of $u,v$ for which
\[
A^{k+1}_{u,v,i}\leq C \E\Big[ |D_v U^{k+1}_{u,\tau_{i+1}}-D_v U_{u,\tau_{i+1}}|^2\Big]+\eta+\alpha A^{k}_{u,v,i}, \qquad \textrm{for }k\geq K^*.
\]
The equation $D_v U_{u,T}=D_v U^k_{u,T}$ allows us to write for
$i=N-1$ and $k\geq K^*$
\[
\sup_{0\leq u,v\leq T} A^k_{u,v,N-1} \leq \eta+\alpha^{k-K^*}\sup_{0\leq u, v\leq T}A^{K^*}_{u,v,N-1}.
\]
As a consequence, (\ref{aux-A}) implies that $\sup_{0\leq u,v\leq
T}A^k_{u,v,N-1}\to 0$ as $k\to \infty$. One can expand the argument
and show recursively that for all $0\leq i\leq N-1$ one has
$\sup_{0\leq u,v\leq T}A^k_{u,v,i}\to 0$. Summing over $i$, one
arrives at
\[
\sup_{0\leq u,v\leq T} \left\{\|D_v U^{k+1}_u-D_v U_u\|_{\cS^2([0,T])}^2+\| D_v V_u^{k+1}-D_v V_u \|_{\cH^2([0,T])}^2\right\}
\stackrel{k\to\infty}{\longrightarrow}
0.
\]
The conclusion is that $(U_{u}, V_{u})$ are indeed Malliavin
differentiable and a version of its Malliavin derivatives is given
by the limit of $(D_v U^k_{u}, D_v V^k_{u})$.

The last statement of our theorem follows from Lemma 5.1 in
\cite{97KPQ}. We write our BSDE
(\ref{appendix:1order-mall-dif-bsde}) for terminal time $t$, apply
the Malliavin derivative operator, and obtain by the quoted Lemma
\begin{align*}
D_v U_{u,t} &= V_{u,t} -\int_v^t D_vV_{u,s}\ud W_s\\
&\quad+\int_v^t \Big[(D_v f)(s,u, U_{u,s}, U_{u,s}) + \big\langle
\nabla f(s,u,U_{u,s}, U_{u,s}),
(D_vU_{u,s},D_vU_{u,s})\big\rangle\Big]\ud s.
\end{align*}
Choosing $v=t$ leads to the desired representation.
\end{proof}

{\bf Acknowledgments:} The authors wish to express their thanks to two referees for their careful reading of the manuscript and several very helpful suggestions.

%\bibliography{quadBSDEnumerics}        % references.bib is the name of the database

\end{document}